\documentclass[reqno,a4paper,11pt]{amsart}
\usepackage[latin1]{inputenc}
\usepackage[english]{babel}
\usepackage{eucal,amsfonts,amssymb,amsmath,amsthm,epsfig,mathrsfs}
\usepackage{dsfont,cancel,soul}
\usepackage{color}
\usepackage{amscd,amsxtra}
\usepackage{enumerate}
\usepackage{latexsym}
\usepackage{bm}

\allowdisplaybreaks

 \newcounter{ipotesi}

\Alph{ipotesi}
 \makeatletter \@addtoreset{equation}{section}
 
 \makeatother \makeatletter

\newtheorem{thm}{Theorem}[section]
\newtheorem{hyp}[thm]{Hypotheses}{\rm}
{\rm}
\newtheorem{lemm}[thm]{Lemma}
\newtheorem{coro}[thm]{Corollary}
\newtheorem{prop}[thm]{Proposition}

\newtheorem{rmk}[thm]{Remark}{\rm}

\newcounter{parentenv}

\newcommand{\R}{{\mathbb R}}

\newcommand{\N}{{\mathbb N}}

\newcommand{\Rd}{\mathbb R^d}

\newcommand{\Cm}{\mathbb C^m}
\newcommand{\Om}{{\Omega}}

\newcommand{\re}{{\rm{Re}}\,}
\newcommand{\im}{{\rm{Im}}\,}

\newcommand{\T}{{\bm T}}

\newcommand{\f}{{\bm f}}

\newcommand{\uu}{{\bm u}}
\newcommand{\A}{\bm{\mathcal A}}

\newcommand{\V}{\mathcal V}
\newcommand{\vv}{{\bm v}}

\newcommand{\af}{\mathfrak a}

\newcommand{\ra}{\rightarrow}
\newcommand{\Vs}{V_{S}}

\newcommand{\C}{\mathbb C}

\renewcommand{\hat}[1]{\widehat{#1}}
\renewcommand{\tilde}[1]{\widetilde{#1}}

\begin{document}

\title[$L^p$-quasicontractiveness and kernel estimates]{\bm{$L^p$}-quasicontractiveness and Kernel estimates for semigroups generated by systems of elliptic operators}
\author[L. Angiuli, L. Lorenzi and E.M. Mangino ]{Luciana Angiuli, Luca Lorenzi, Elisabetta M. Mangino}
\address{L.A. \& E.M.M.:  Dipartimento di Matematica e Fisica ``Ennio De Giorgi'', Universit\`a del Salento, Via per Arnesano, I-73100 LECCE, Italy}
\address{L.L.: Dipartimento di Scienze Matematiche, Fisiche e Informatiche, Plesso di Mate\-matica, Universit\`a degli Studi di Parma, Parco Area delle Scienze 53/A, I-43124 PARMA, Italy}
\email{luciana.angiuli@unisalento.it}
\email{luca.lorenzi@unipr.it}
\email{elisabetta.mangino@unisalento.it}
\thanks{The authors are members of G.N.A.M.P.A. of the Italian Istituto Nazionale di Alta Matematica (INdAM). This paper is based upon work from the project PRIN2022 D53D23005580006 ``Elliptic and parabolic
problems, heat kernel estimates and spectral theory''}

\begin{abstract}
This paper focuses on systems of strongly coupled elliptic operators whose coefficients may be unbounded and are defined on a domain $\Omega \subseteq \mathbb{R}^d$. It is shown that a quasi-contractive semigroup in $L^p$-spaces can be associated with such operators for values of $p$ belonging to an interval that contains $2$ as an interior point. Then, under refined assumptions and considering systems of elliptic operators coupled up to first order, new kernel estimates are established with respect to a distance function that accounts for the growth of the diffusion coefficients and the potential term at infinity.
\end{abstract}
\maketitle

\section{Introduction}

In this paper, we investigate the $L^p$-quasicontractiveness of strongly continuous semigroups on $L^2(\Om;\C^m)$, which are associated with closed sesquilinear forms, 
and the consequent kernel estimates one can derive. Here, $\Om\subset \R^d$ is an open set and $d,m\in \N$.

In our analysis, we consider operators formally defined by
\begin{equation}
\label{op_ke-intro}
\A\uu=\sum_{h,k=1}^d D_h(q_{hk}D_k \uu)
+\sum_{h,k=1}^d D_h(A^{hk}D_k \uu)-\sum_{h=1}^d  B^h D_h \uu+\sum_{h=1}^d D_h( C^h\uu)-V\uu
\end{equation}    
on smooth vector-valued functions $\uu:\Omega\to \C^m$. Under suitable assumptions on the coefficients of $\A$, which are allowed to be unbounded, (see Sections \ref{sect-2}, \ref{sect-3} and \ref{sect-4}), we prove that 
the operator $\A$ generates a strongly continuous analytic semigroup $(\T_2(t))_{t \ge 0}$ in $L^2(\Om;\C^m)$. 
The purpose of this paper is twofold: first, to provide conditions which guarantee that $(\T_2(t))_{t \ge 0}$  can be extended to a bounded strongly continuous and analytic semigroup in $L^p(\Om;\C^m)$
for $p$ in a suitable interval containing $2$ as an interior point; second to prove Gaussian-type estimates for the heat kernel of the extended semigroup whenever such an extension exists for every $p \in ]1,\infty[$,

The problem of characterizing the $L^p$-contractivity or quasi-contractivity of semigroups generated by elliptic operators in divergence-form has a long history. Early results in this direction are due to Agmon, Douglis and Nirenberg in \cite{ADN}. A summary of references on this topic in the scalar case can be found in \cite{CM}, where elliptic operators in divergence-form with complex coefficients are considered. For instance, in \cite{CM04} the authors provide an algebraic characterization of the $L^p$-dissipativity of a quadratic form and consequently of the $L^p$-contractivity of the associated semigroup, in \cite{CD}, the notion of $p$-ellipticity is introduced to refine the results in \cite{CM} and prove the $L^p$-contractivity of semigroups associated to pure second-order scalar elliptic operators with complex coefficients.
Clearly, studying this kind of operators is equivalent to analyzing vector-valued operators with real coefficients.

In our situation, there are two main difficulties to be overcome. 
On one hand, the possible unboundedness of the coefficients of the operator $\A$, even in the scalar case, represents a significant difference with respect to the classical case of bounded coefficients. 
On the other hand, the presence of non-zero matrices $A^{hk}$ yields   the failure of the maximum modulus principle, and consequently of $L^\infty$-contractivity for the semigroup associated with $\A$. 
This prevents us from deducing $L^p$-estimates by interpolation,  a strategy that works well  instead in the weakly-coupled case. In the latter situation $L^\infty$-contractivity is tipically established in two different ways: using the Beurling-Deny criterion (see \cite[Theorem 1.3.3]{Dav} or \cite{Ouh1}, see also \cite[Section 2.5]{Ou}), or by proving a domination property $|\T_2(t)\f|\le T_2(t)|\f|$ for every $t \ge 0$ and $\f \in L^2(\Om;\C^m)$, where $(T_2(t))_{t\ge 0}$ is a scalar semigroup on $L^2(\Om)$ which enjoys good properties (see \cite{Ouh2}).
This semigroup domination plays an important role in several situations and allows the transfer of many properties  from $T_2(t)$ to $\T_2(t)$, such as heat kernel estimates, compactness and regularizing (summability improving) properties.

%As already observed, since we cannot expect the $L^\infty$-contractivity for $\T_2(t)$ we argue differently and we base our results on the theory of sesquilinear forms. The classical Lax-Milgram's theorem is the first step to deduce that the opposite of the operator defined in \eqref{op_ke-intro} generates a strongly continuous and analytic semigroup $(\T_2(t))_{t \ge 0}$ in $L^2(\Om;\C^m)$.

 The key tool we use to prove that $(\T_2(t))_{t \ge 0}$ extends consistently to a bounded semigroup $(\T(t))_{t \ge 0}$ in $L^p(\Om;\Cm)$  for some given $p \in (1,\infty)$ is a result due to Nittka (see \cite{Nit}) which provides a criterion for the $L^p$-contractivity of a $C_0$-semigroup on $L^2(\Om)$. More precisely, we take advantage of a more general version proved in \cite{ATS}, which deals with the vector-valued space $L^2(\Om;H)$, where $H$ is a Hilbert space. In that paper, the authors apply this criterion (in the case  $H=\C^m$) to purely second-order systems of elliptic operators with bounded coefficients, and show the existence of an interval $[p_-,p_+]$ with $1<p_-<2<p_+<\infty$ such that the semigroup generated in $L^2(\Om;\C^m)$ extends to a contractive $C_0$-semigroup on $L^p(\Om;\C^m)$ for all $p \in [p_-,p_+]$.

Nittka's criterion, stated in Theorem \ref{Nittka}, characterizes the $L^p$-contractivity of a semigroup generated on $L^2(\Omega;\C^m)$ and associated with a form $\af:\V\times \V\to \C$, where $\V$ is a Hilbert space, densely embedded in $L^2(\Om;\C^m)$, through two conditions. The first one requires that $\V$ be invariant under the orthogonal projection $P_p$ from $L^2(\Om;C^m)$ onto the closed convex set $\{\uu\in L^2(\Om;\C^m):\|\uu\|_p\le 1\}$. The second condition requires that $\re \af(\uu, |\uu|^{p-1}{\rm sign}\uu)\ge 0$ for every $\uu\in \V$ such that $|\uu|^{p-1}{\rm sign}\uu\in \V$. The role of the space $\V$ is to encode the boundary conditions. Although the dependence on $\V$ does not explicitly appear in the formal expression of $\A$, it is clear that two different spaces $\V$ and $\mathcal{W}$ give rise to two different operators $\A_\V$ and $\A_{\mathcal W}$, as they correspond to different boundary conditions.
 
In our setting, $\V$ is an abstract Hilbert space satisfying suitable assumptions (see Hypothesis \ref{V}), which turns out to be invariant with respect to $P_p$ thanks to the $L^\infty$-contractivity and hence the $L^p$-contractivity of the semigroup associated with the form 
form $\mathfrak{a}^0$ defined by 
\begin{align*}
\mathfrak a^0 (\uu,\vv)= \int_{\Omega} \bigg[\sum_{i=1}^m(Q\nabla u_i, \nabla v_i) + (\Vs\uu,\vv)\bigg ] dx, \qquad \uu, \vv \in\mathcal V,
\end{align*}
where $Q=(q_{hk})_{h,k=1,\ldots,d}$ and $V_S$ is the symmetrized part of the matrix of $V$. 
Moreover, the sign condition on the real part of $\af(\uu,|\uu|^{p-1}{\rm sign} \uu)$ allows us to apply Nittka's criterion to show that $(\T_2(t))_{t \ge 0}$ extends consistently to a semigroup $(\T_p(t))_{t \ge 0}$ on $L^p(\Om;\C^m)$ for $p$ lying in an interval that depends explicitly on the coefficients of the operator $\A$. In this case, the operator norm of $\T_p(t)$ is estimated uniformly with respect to $p$ in such interval. 

A slight different set of assumptions on the drift and the potential terms allows, once again via Nittka's criterion, to estimate the operator norm of $\T_p(t)$ by a function depending on $p$, when $p$ belongs to an interval depending essentially on the terms $A^{hk}$ in \eqref{op_ke-intro}. This interval coincides with the right-halfline $]1,\infty[$ whenever the matrix-valued functions $A^{hk}$ vanish for every $h,k=1, \ldots,d$.

Next, by refining the assumptions so that the $L^p$-contractivity for $(\T(t))_{t \ge 0}$ (we drop the dependence on $p$ thanks to the consistence of the semigroups $(\T_p(t))_{t \ge 0}$ in the $L^p$-scale) holds true for every $p \in ]1,\infty[$, we take advantage of the aforementioned estimates to prove generalized Gaussian estimates for the heat kernel $\bm{k}(t,\cdot, \cdot)=(k_{ij}(t,\cdot,\cdot))_{i,j=1,\ldots, m}$ of the semigroup $(\T(t))_{t \ge 0}$ in terms of a distance $d_{Q,V,\beta}$ which depends on $Q$, $V$ and a positive constant $\beta$ (see \eqref{distanza}). This metric is equivalent to the Euclidean one for instance when $Q$ and $V$ are bounded. We point out that, to our knowledge, this is the first result establishing kernel estimates for semigroups generated by  elliptic operators coupled up to first order. Other results concerning Gaussian-type estimates for vector-valued semigroups associated to elliptic operators with unbounded coefficients can be found, for instance, in \cite{ALP,ALMR1} 
See also \cite{ALLP,AngLorMan1,ALMR,AngLorMan,AngLorMan2,AngLorPal,KmR} and the reference therein for further results on vector-valued semigroups in the classical $L^p$-setting. 

% To prove Gaussian-type estimates, we use the so-called  ``Davies trick" which allows us to deduce upper bounds of the form 
% \begin{equation}\label{est_intro}
% |k_{ij}(t,x,y)|\le c_1(t)e^{-c_2(t)(\psi(x)-\psi(y))^{\frac{2\beta+2}{2\beta+1}}},\qquad\;\, i,j=1, \ldots, m
% \end{equation} 
% for any $t>0$ and almost every $x, y\in\Omega$. Here $\beta$ is a positive real number (see Hypotheses \ref{base-2}), while the positive functions $c_1$ and $c_2$ exhibit \textcolor{red}{a rather explicit} behaviour near $0$.  The function $\psi$ is a smooth, real-valued function belonging to a class of test functions satisfying the condition $(Q\nabla \psi, \nabla \psi)\le \lambda_V^{\frac{\beta}{\beta+1}} $ almost everywhere in $\Om$, where $\lambda_V$ denotes the minimum eigenvalue of $V_S$. 
% We point out that in the original formulation of Davies trick, the supremum is taken over smooth functions $\psi$ such that $(Q\nabla \psi, \nabla \psi)\le 1 $ almost everywhere in $\Om$.
% Taking the supremum over all admissibile functions $\psi$ in \eqref{est_intro} yields an estimate of $k_{ij}(t,x,y)$ in terms of the function $(t,x,y)\mapsto\exp(-c_2(t)(d_{Q,V,\beta}(x,y))^{\frac{2\beta+2}{2\beta+1}})$, where $d_{Q,V, \beta}$ is a distance depending on $Q$ and $V$ and the positive constant $\beta$, see \eqref{distanza} for its definition. 
% When $Q$ and $V$ are unbounded, this distance is generally not equivalent to the Euclidean one. However, in some cases equivalence holds, resulting in an explicit estimate for the kernel $\bm{k}(t, \cdot,\cdot)$.

The plan of the paper is as follows. in Section \ref{sect-2} we collect the standing assumptions, define the form $\af$ we are interested in, and derive the main preliminary results needed in the subsequent sections. We also recall some relevant results from \cite{ATS} and \cite{Nit}. In Section \ref{sect-3}, after establishing the generation of a strongly continuous analytic semigroup $(\T_2(t))_{t\ge 0}$ associated with $\af$ in $L^2(\Om;\C^m)$, we present two different $L^p$-estimates for $(\T_2(t))_{t \ge 0}$ and prove that it extends consistently across the $L^p$-scale. In Section \ref{sect-4}, we establish the kernel estimates. As a byproduct, we also obtain the existence of a bounded $H^\infty$-calculus in $L^p(\Omega;\Cm)$ for the first-order coupled operator $\A$, when the diffusion and drift coefficients are bounded. Finally, Section \ref{sect-5} provides some examples of operators of the form \eqref{op_ke-intro} and the corresponding semigroups to which our results apply.

%\begin{itemize}
    %\item 
%Richiami di risultati noti (Nittka)
%\item 
%costruzione semigruppo associato a $\sum_{i,j=1}^dD_i(q_{ij}D_j\uu)-V\uu$ con le forme.
%\item 
%Criterio vettoriale di Beurling-Deny per la $L^{\infty}$-contrattivit\`a del semigruppo associato all'operatore al punto precedente.
%\item 
%$L^p$-contrattivit\`a del semigruppo.
%\end{itemize}

\bigskip

{\bf Notation}
We denote  by $(\cdot, \cdot)$ and by $|\cdot|$, respectively, the Euclidean inner product and the Euclidean norm in $\mathbb K^m$, where $m\in\N$ and
$\mathbb K=\mathbb R$ or $\mathbb K=\mathbb C$. 

For every ($m\times m$)-matrix $M$, we denote by $M^T$ its transposed and by $M_{S}$ its symmetric part i.e., $M_{S}=\frac{1}{2}(M+M^T)$. If $M$ is a symmetric matrix, by $\lambda_M$ we denote its minimum eigenvalue.
For every $\theta^1,\ldots,\theta^d\in \C^m$ ($d\in\N$) and every $(d\times d)$-matrix $Q$, with entries $q_{ij}$ ($i,j=1,\ldots,d$), we set
$\mathcal{Q}(\theta^1,\ldots,\theta^d)=
\sum_{i=1}^m\sum_{h,k=1}^dq_{hk}(x)\theta^k_i\overline\theta^h_i$.

Vector-valued functions are displayed in bold style. Given an open set $\Omega$ and a function $\uu: \Omega \subseteq \Rd\to {\mathbb K}^m$, we denote by $u_k$ its $k$-th component. 
Morever we  set
${\rm sign}\hskip 1pt\uu=|\uu|^{-1}\uu\chi_{\{\uu\not={\bf 0}\}}$. 

For every $p\in [1,\infty[\,\cup\{\infty\}$\,,  $L^p(\Omega; \mathbb K^m)$  denotes the classical vector-valued Lebesgue space, endowed with the norm
$\|\f\|_p=(\int_{\Om} |\f(x)|^pdx)^{1/p}$, if $p<\infty$, and $\|\f\|_{\infty}=\operatorname{ess sup}_{x\in\Omega}|\f(x)|$, if $p=\infty$.
$W^{1,p}(\Omega; \mathbb K^m)$ is the classical vector-valued Sobolev space, i.e., the space of all functions $\uu\in L^p(\Om; \mathbb K^m)$ whose components have distributional first-order partial derivatives which belong to $L^p(\Om; \mathbb K^m)$. $W^{1,p}(\Om; \mathbb K^m)$ is endowed with its natural norm. Moreover, $W^{1,p}_0(\Om;\mathbb K^m)$ is the closure of 
$C^{\infty}_c(\Om;\mathbb K^m)$ with respect to the $W^{1,p}(\Omega;\mathbb K^m)$-norm, where 
$C^{\infty}_c(\Om;\mathbb K^m)$ denotes the set all the vector-valued functions which have compact support in $\Om$ and are infinitely many times differentiable.
When $m=1$, we simply write $L^p(\Om)$ and $C^{\infty}_c(\Om)$. If $X(\Om;\mathbb{K}^m)$ is $L^p(\Omega;\mathbb K^m)$ or $W^{1,p}(\Omega;\mathbb K^m)$, then we use the notation $X_{\rm{loc}}(\Om;\mathbb{K}^m)$ to denote the set of functions which belong to $X(K;\mathbb{K}^m)$ for every compact set $K\subset \Omega$.

\section{Preliminaries}
\label{sect-2}

We start by recalling the well-known Lax-Milgram theorem about operators associated with sesquilinear forms.

\begin{thm}\label{TH:coerc}
Let $\mathcal H, \mathcal V$ be complex  Hilbert spaces with $\mathcal V$ densely and continuously embedded in $\mathcal H$ and let $\af:\mathcal V\times\mathcal V\rightarrow \C$ be a sesquilinear form. If $\af$ is continuous with respect to the norm $\|\cdot\|_\mathcal V$ of $\mathcal V$ and  elliptic,  i.e., there exist $\omega\in\R$, $\mu>0$ and a dense subspace $D$ of $\mathcal V$ such that 
\begin{equation*}
\re\af(u,u) + \omega \|u\|^2_{\mathcal H} \geq\mu\|u\|^2_{\mathcal V}, \qquad\;\, u\in D, 
\end{equation*}
then the opposite of the linear operator $A: D(A) \subseteq \mathcal H \rightarrow \mathcal H$, defined by 
\begin{eqnarray*}
\left\{
\begin{array}{l}
D(A)=\{ u\in {\mathcal V}\, \mid\, \exists w\in \mathcal H \textit{ s.t. }  (w, \varphi)=\af(u,\varphi) \  \forall \varphi\in \mathcal V\},\\[1mm]
Au=w,\quad u\in D(A),
\end{array}
\right.
\end{eqnarray*}
generates a strongly continuous analytic  semigroup $(T(t))_{t\ge 0}$ on $\mathcal H$.
Moreover, if $\omega_0=\inf\{\omega\in\R: \re \af(u,u)+\omega \|u\|^2_{\mathcal H}\geq 0 \ \forall u\in \mathcal V\}$, then 
\begin{eqnarray*}
\|T(t)\|_{\mathcal L(\mathcal H)}\leq e^{\omega_0 t}, \qquad t\geq 0.
\end{eqnarray*}
\end{thm}

\begin{rmk}
\label{rem-2.2}
{\rm 
\begin{enumerate}[\rm (i)]
\item 
An immediate consequence of Theorem \ref{TH:coerc} is the fact that, if the form $\af$ is accretive (i.e., $\re\af(u,u)\geq0$ for every $u\in\mathcal V$), densely defined, continuous and elliptic, then the opposite of the operator associated to $\af$, generates a contractive semigroup.
\item
If $\af$ is densely defined, continuous and elliptic, then the same holds for  the adjoint form $\af^*:\mathcal V \times \mathcal V\rightarrow \C$, defined by
\begin{eqnarray*}
\af^*(u,w)=\overline{\af(w,u)}, \qquad\;\, u, w\in\mathcal V.
\end{eqnarray*}
Moreover, the operator associated with $\af^*$ is the adjoint operator $A^*$ of the operator $A$ associated with $\af$. In particular, the semigroup generated by $A^*$ is the adjoint of the semigroup $(T(t))_{t\ge 0}$ and 
$\|T(t)^*\|_{\mathcal{L}(\mathcal{H})}\le e^{\omega_0t}$ for every $t\in [0,\infty[$\,.
\end{enumerate}
}
\end{rmk}

We will be interested in the study of $L^p$-contractivity  of the semigroup $(\T(t))_{t\ge 0}$ associated with a system of elliptic operators. For this purpose, we will take advantage of  
 the Nittka's criterion for $L^p$-contractivity, adapting the more general setting of \cite[Theorem 1.1]{ATS} and \cite[Theorem 4.1]{Nit}, to the case $L^p(\Omega; \C^m)$.
 
\begin{thm}[Nittka's Criterion]\label{Nittka}
Fix $p\in ]1, \infty[$\,, let 
\begin{eqnarray*}
C_p=\{ \uu\in L^2(\Omega; \C^m)\cap L^p(\Omega; \C^m):  \|\uu\|_p\leq 1\}
\end{eqnarray*}
and let $P_p$ be the orthogonal projection of $L^2(\Omega; \C^m)$ on $C_p$. Further, consider a  densely embedded Hilbert subspace $\V$ of $L^2(\Omega; \C^m)$ and let $\af:\V\times\V\to\C$ be a continuous elliptic sesquilinear form, with associated operator $A:D(A)\subseteq\V\to\V$. Finally, let  $(\T(t))_{t\ge 0}$ be the semigroup generated in $L^2(\Omega; \C^m)$ by the operator $-A$. 
Then, the following conditions are equivalent:
\begin{enumerate}[\rm (i)]
\item
$\|\T(t)\uu\|_p\leq\|\uu\|_p$ for every $\uu\in L^2(\Omega; \C^m) \cap L^p(\Omega; \C^m)$;
\item
$P_p(\V)\subseteq \V$ and $\re \af(\uu, |\uu|^{p-1}{\rm sign}\hskip 1pt \uu)\geq 0$ for every $\uu\in \V$ such that $ |\uu|^{p-1}{\rm sign}\uu\in \V$.
\end{enumerate}
\end{thm}

\begin{rmk}
\label{rem-2.4}
{\rm We point out that a slight modification to the Nittka's Criterion allows to characterize the $L^p$-quasi contractivity of the semigroup $(\T(t))_{t\ge 0}$ in terms of the conditions stated in Theorem \ref{Nittka}(ii) just replacing the sesquilinear form $\af(\uu,\vv)$ with the form $\af(\uu,\vv)+\omega (\uu,\vv)$ for some $\omega \in \R$. In this case, the operator associated to $\af+\omega$ is $-A-\omega$ and the semigroup $(\T(t))_{t\ge 0}$ is replaced by $(e^{-\omega t}\T(t))_{t\ge 0}$, whose contractivity in $L^p$ is equivalent to the $L^p$-quasi contractiveness of $(\T(t))_{t\ge 0}$.
}
\end{rmk}

We now introduce the standing assumptions on the coefficients of the operator $\A$, defined in \eqref{op_ke-intro}.

\begin{hyp}\label{base}
\begin{enumerate}[\rm (i)]
\item 
the entries $q_{hk}$ $(h,k=1,\ldots,d)$ of the 
matrix-valued function $Q$ belong to $L^\infty_{{\rm loc}}(\Omega)$; for  almost every  $x\in\Omega$,  $Q(x)$ is symmetric and the function $\lambda_Q:\Omega\to ]0,\infty[$ has   strictly positive essential infimum on compact subsets of $\Omega$;
\item
the matrix-valued function $V$ belongs to $L^{\infty}_{{\rm loc}}(\Omega; \R^{m\times m})$ and $\operatorname{ essinf}_{x\in\Omega}{\lambda_{V_S}(x)}=:v_0>0$. 
Further, there exists a positive constant $c_0$ such that 
\begin{eqnarray*}
 |{\rm Im}(V\xi, \xi)|\le c_0 {\rm Re}(V\xi,\xi),\quad\;\,\end{eqnarray*}
 for every $\xi\in\C^m$ and a.e. in $\Omega$;

\item 
the matrix-valued functions $A^{hk}$ $(h,k=1,\dots,d)$ belong to $L^\infty_{{\rm loc}}(\Omega; \R^{m\times m} )$ and  there exists a positive constant $\kappa_A$ such that
\begin{align}
\label{realLegendrefunc}
0\leq  &{\rm Re}\sum_{h,k=1}^d (A^{hk}\theta^k, \theta^h) \leq \kappa_A\mathcal{Q}(\theta^1,\ldots,\theta^d),\\
 \label{comLeg}
 &\bigg | {\rm Im}\sum_{h,k=1}^d (A^{hk}\theta^k, \theta^h)\bigg |\leq \kappa_A \mathcal{Q}(\theta^1,\ldots,\theta^d)
  \end{align}
  for every $\theta^1,\ldots,\theta^d\in \Cm$ and almost everywhere in $\Om$;
  %\luca{controllare dove serve $\kappa_A<1$}
\item the matrices  $B^h, C^h$, $h=1,\dots,d$ and $W$ belong to $L^\infty_{\rm loc}(\Om;\R^{m\times m})$ and
there exist non negative  constants $\kappa_B, \kappa_C, \kappa_W$ and positive constants $\gamma$ and $C_\gamma$  such that for every $\theta^1,\ldots,\theta^d, \xi, \eta\in \Cm$:
\begin{align}\label{est_1}&\bigg|\sum_{h=1}^d (B^h \theta^h, \eta) \bigg| \leq \kappa_B\big(\mathcal{Q}(\theta^1,\ldots,\theta^d)\big )^\frac{1}{2}(\gamma  (\Vs\eta, \eta) + C_\gamma |\eta|^2)^{\frac{1}{2}},
\end{align}
\begin{align}\label{est_2}
&\bigg|\sum_{h=1}^d (C^h \eta,\theta^h) \bigg| \leq  \kappa_C\big(\mathcal{Q}(\theta^1,\ldots,\theta^d)\big )^\frac{1}{2}(\gamma(\Vs\eta, \eta) + C_\gamma |\eta|^2)^{\frac{1}{2}},
\end{align}
\begin{align}\label{est_3}
&|(W\xi, \eta) | \leq \kappa_W(\gamma (\Vs\xi, \xi) + C_\gamma |\xi|^2)^\frac 1 2(\gamma (V_S\eta, \eta) + C_\gamma|\eta|^2)^\frac{1}{2}
\end{align}
almost everywhere in $\Om$;
\item $K:=4\left(\frac{1}{\gamma}-\kappa_W\right)-(\kappa_B+\kappa_C)^2>0$.
\end{enumerate}
\end{hyp}

\begin{rmk}[Proposition A.1 \& Corollary A.2 in \cite{AngLorMan}]
$\;$\\[-4mm]
\rm{
\begin{enumerate}[\rm (i)]
\item 
Under Hypothesis \ref{base}(ii), it holds that
%\begin{eqnarray*}
%{\rm Re}(V(x)\xi,\xi)=(\Vs \xi, \xi)\ge \lambda_V(x)|\xi|^2
%\end{eqnarray*}
%and
\begin{equation}
|(V\xi, \eta)|\le (1+c_0) (\Vs\xi, \xi)^{\frac{1}{2}} (\Vs\eta, \eta)^{\frac{1}{2}}
\label{cond-c0}
\end{equation}
almost everywhere in $\Omega$ and for every $\xi, \eta \in \C^m$. 
\item 
Estimates \eqref{realLegendrefunc} and \eqref{comLeg} yield the inequality 
\begin{equation}\label{CS}
\bigg |\sum_{h,k=1}^d (A^{hk}\theta^k, \eta^h) \bigg | \leq 2 \kappa_A\mathcal{Q}(\theta^1,\ldots,\theta^d)^{\frac{1}{2}}
\mathcal{Q}(\eta^1,\ldots,\eta^d)^{\frac{1}{2}},
\end{equation}
which holds true almost everywhere in $\Omega$ and for  every $\theta^1,\dots,  \theta^d$, $\eta^1,\ldots, \eta^d \in \C^m$.
\item 
Condition \eqref{est_3} on
the matrix $W$ is satisfied, for instance, if 
\begin{eqnarray*} 
0\leq \re(W\eta, \eta)\leq \frac{\kappa_W}{2}( \gamma(\Vs \eta, \eta) + C_\gamma|\eta|^2)
\end{eqnarray*}
and 
\begin{eqnarray*} 
\big|{\rm Im}\hskip 1pt (W\eta, \eta)\big| \leq \frac{\kappa_W}{2} (\gamma(\Vs\eta, \eta) + C_\gamma|\eta|^2)
\end{eqnarray*}
almost everywhere in $\Omega$ and for every $\eta\in \C^m$.
\end{enumerate}
}
\end{rmk}
\medskip

Now, we introduce the set $D_{Q,V}$ consisting of all functions 
\begin{eqnarray*} 
D_{Q,V}=\{\uu \in W_{\rm loc}^{1,2}(\Om; \mathbb{C}^m): \uu,\ \Vs^\frac 1 2 \uu,\  Q^{\frac{1}{2}}D_j \uu\in L^2(\Om;\Cm) \text{ for every }j=1,\dots,d   \},
\end{eqnarray*}
endowed with the inner product
\begin{eqnarray*} 
\langle \uu, \vv\rangle_{Q, V}=\int_{\Om} \bigg [\sum_{i=1}^m (Q\nabla u_i, \nabla v_i) +(\Vs\uu,\vv) \bigg ]dx, \qquad\;\,  \uu, \vv\in D_{Q,V},
\end{eqnarray*}
and the corresponding norm 
$\|\uu\|_{Q,V} =\langle\uu,\uu\rangle_{Q,V}^{\frac{1}{2}}$ for every $\uu\in D_{Q,V}$.

\begin{lemm}
The space $D_{Q,V}$, endowed with inner product $\langle \cdot, \cdot\rangle_{Q,V}$, is a Hilbert space.
\end{lemm}

\begin{proof}
Let $(\uu_n)$ be a Cauchy sequence in $D_{Q,V}$, which means that $(\Vs^{\frac{1}{2}}\uu_n)$ and $(Q^{\frac{1}{2}}D_j\uu_n)$ ($j=1,\ldots,d$) are Cauchy sequences in $L^2(\Omega;\C^m)$. 
Thanks to Hypotheses \ref{base}, $(\uu_n)$ is a Cauchy sequence in $W_{{\rm loc}}^{1,2}(\Om; \C^m)$ and in $L^2(\Om; \C^m)$. Hence,  there exists a function $\uu \in W_{\rm loc}^{1,2}(\Om;\Cm)\cap L^2(\Omega;\C^m)$ such that $\uu_n$ converges to $\uu$ in $L^2(\Omega;\C^m)$, and,   up to a subsequence,  $Q^{\frac{1}{2}}D_j\uu_n$ converges  to  $Q^{\frac{1}{2}}D_j\uu$, for every $j=1,\ldots,d$, and $\Vs^\frac 1 2\uu_n$  converges to   $\Vs^\frac 1 2 \uu$ almost everywhere in $\Om$. Since $(Q^{\frac{1}{2}}D_j\uu_n)$ is a Cauchy sequence in $L^2(\Om;\C^m)$ for every $j=1,\ldots,d$, it converges to $Q^{\frac{1}{2}}D_j\uu$ in $L^2(\Omega;\C^m)$ and, analogously, 
$\Vs^\frac 1 2\uu_n$
converges to $\Vs^\frac 1 2\uu$ in $L^2(\Omega;\Cm)$. Thus, we conclude that $\uu_n$ converges to $\uu$ in $D_{Q,V}$, which is a Hilbert space.
\end{proof}

Now, we fix the space ${\mathcal V}$.

\begin{hyp}\label{V}
\begin{enumerate}[\rm (i)]
\item
$\mathcal V$ is a closed subspace of $D_{Q,V}$ which contains $C_c^\infty(\Omega;\C^m)$;
\item 
for every $\uu\in \V$ the function $(1\wedge |\uu|){\rm sign }\,\uu$ belongs to $\V$.
\end{enumerate}
\end{hyp}

\begin{lemm}\label{ELI}
If $\mathcal{V}=\overline{C_c^\infty(\Omega; \Cm)}^{D_{Q,V}}$, then 
$(|\uu|\wedge 1){\rm sign}\hskip 1pt\uu$ belongs to ${\mathcal V}$ for every $\uu\in{\mathcal V}$.
\end{lemm}

\begin{proof}
We begin by proving that
${\mathcal V}= \overline{\bigcup\{W_0^{1,2}(\Omega',\Cm): \Omega' \text{ open,} \ \Omega'\Subset \Omega\}}^{D_{Q,V}}$.
For this purpose, it suffices to show that $W_0^{1,2}(\Omega', \Cm)\subseteq \overline{C_c^\infty(\Omega; \Cm)}^{D_{Q,V}}$ for every $\Omega'\Subset\Omega$. So, fix $\f\in W_0^{1,2}(\Omega',\Cm)$  and let $(\rho_n)$ be a sequence of mollifiers such that $\rho_n(x)=0$ if $|x|\geq \frac 1 n$, for every $n\in\N$. 
Then, the convolution $\f*\rho_n$ tends to $\f$ in $W^{1,2}(\Omega;\Cm)$ as $n$ tends to $\infty$ and there exists $n_0\in\N$ such that for $n\geq n_0$ 
\begin{eqnarray*} 
\text{supp}(\f*\rho_n)\subseteq \overline{\Omega'+ B(0,n^{-1})}\subseteq \overline{\Omega'+ B(0,n_0^{-1})}\subseteq \Omega, 
\end{eqnarray*}
so that $\f*\rho_n\in C_c^\infty(\Omega;\Cm)$.
Finally, observe that, thanks to Hypotheses \ref{base}, the convergence  of $(\f*\rho_n)_n$ with respect to $\|\cdot\|_{Q,V}$ is equivalent to the convergence in the $W^{1,2}$-norm on $\overline{\Omega'+ B(0,n_0^{-1})}$, since
$Q$ and $V$ are locally bounded in $\Omega$ and $\lambda_Q$, $\lambda_V$ are locally strictly positive functions. Thus, $\f\in {\mathcal V}$.

To conclude the proof, we fix $\uu\in  {\mathcal V}$ and a sequence $(\uu_n)\subset C_c^\infty(\Om; \Cm)$, which converges to $\uu$ in $D_{Q,V}$ and  $V_S^{\frac 1 2}\uu_n$,
$Q^{\frac 1 2}\nabla\uu_n$ converge almost everywhere to
$V_S^{\frac 1 2}\uu_n$, 
$Q^{\frac 1 2}\nabla\uu_n$, respectively, as $n$ tends to $\infty$. Without loss of generality, we can assume that $|V_S^{\frac 1 2}\uu_n|\leq \psi$ and $|Q^\frac 1 2 \nabla \uu_n| \leq \psi$ for some
$\psi\in L^2(\Om)$.

By Lemma \ref{lemma-A1} and taking into account that $\uu_n$ is compactly supported in $\Omega$ for every $n\in\N$, we infer that the function $\vv_n=(|\uu_n|\wedge 1)\text{sign}\hskip 1pt\uu_n$ is in $W_0^{1,2}(\Omega', \Cm)$ for some bounded open subset $\Om'$ with closure contained in $\Om$. Due to the characterization of ${\mathcal V}$ proved above, for every $n\in\N$, $\vv_n$ belongs to ${\mathcal V}$ and 
\eqref{form-A1} guarantees that
\begin{eqnarray*} 
D_k\vv_n= 
\begin{cases}
\displaystyle
-\frac{\uu_n}{|\uu_n|^2}D_k |\uu_n| + \frac{D_k\uu_n}{|\uu_n|}, \qquad  & |\uu_n|>1,\\[3mm]
D_k\uu_n, \qquad &|\uu_n|\leq 1,
\end{cases}
\end{eqnarray*}
for every $k=1,\ldots,d$.

Let us prove that $(\vv_{n})$ converges to $(|\uu|\wedge 1)\text{sign}\hskip 1pt\uu$ in $D_{Q,V}$. This will imply that $(|\uu|\wedge 1)\text{sign}\hskip 1pt\uu\in\mathcal{V}$.
Since $\lambda_V$ and $\lambda_Q$ are positive almost everywhere in $\Omega$, it follows that $\uu_n$ converges to $\uu$ almost everywhere in $\Om$ and  $\nabla\uu_n$ converges to $\nabla\uu$ almost everywhere in $\Om$ as $n$ tends to $\infty$.
Consequently, $(\vv_n)$ converges  to $(|\uu|\wedge 1)\text{sign}\hskip 1pt\uu$ almost everywhere in $\Om$ and 
\begin{eqnarray*} 
D_k\vv_n
\rightarrow
\begin{cases} 
\displaystyle
-\frac{\uu }{|\uu|^2}D_k |\uu| + \frac{D_k\uu}{|\uu|}, \qquad  & |\uu|>1,\\[1mm]
D_k\uu, \qquad &|\uu|\leq 1,
\end{cases}
\end{eqnarray*}
for every $k=1,\ldots,d$.

Finally, using the estimates $|V_S^{\frac 1 2 }\vv_n|\leq \psi$  and $|Q^\frac 1 2\nabla \vv_n|\leq \psi$,  we can apply Lebesgue's dominated convergence theorem and get that $(\vv_n)$  converges to $(|\uu|\wedge 1)\text{sign}\hskip 1pt\uu$ in $D_{Q,V}$.
\end{proof}
 
\begin{prop}\label{pcontr}
Under Hypotheses $\ref{base}(i)$-$(ii)$ and $\ref{V}$,
for every $p\in ]1,\infty[$, the projection $P_p$ in Theorem $\ref{Nittka}$ leaves the space $\V$ invariant.
\end{prop}

\begin{proof} 
We introduce the sesquilinear form $\mathfrak{a}^0$ on $\V\times\V$, defined by 
\begin{align*}
\mathfrak a^0 (\uu,\vv)= \int_{\Omega} \bigg[\sum_{i=1}^m(Q\nabla u_i, \nabla v_i) + (\Vs\uu,\vv)\bigg ] dx, \qquad \uu, \vv \in\mathcal V,
\end{align*}
which is clearly densely defined since $\mathcal V$ is densely embedded in $L^2(\Omega; \C^m)$,  accretive, continuous and elliptic on $L^2(\Om;\C^m)$, since $\mathfrak a^0(\uu,\uu) = \langle \uu,\uu \rangle_{Q,V}$ for every $\uu\in\V$.
Hence, to $\af^0$  we can associate an operator $(\A^0_\V, D(\A^0_\V))$, defined as in Theorem \ref{TH:coerc},
%\begin{eqnarray*}\left\{
%\begin{array}{l}
%D(\A^0_\V)=\{ \uu\in {\mathcal V}:  \exists \ww\in L^2(\Omega; \C^m) \textit{ s.t. }  (\ww, \varphi)=\mathfrak a_0(\uu,\varphi) \  \forall \varphi\in \mathcal V\},\\[1mm]
%\A^0_\V\uu=\ww,\quad \uu\in D(\A^0_\V),
%\end{array}
%\right.
%\end{eqnarray*}
such that $-\A^0_\V$ generates a contractive strongly continuous analytic  semigroup $(\T^0_{\V}(t))_{t\ge 0}$ on $L^2(\Om; \C^m)$.

To conclude the proof, it suffices to show that the semigroup $(\T^0_{\V}(t))_{t\ge 0}$ extends to the $L^p$-scale with a strongly continuous contractive semigroup and then apply Nittka's criterion. For this purpose, we begin by proving that the semigroup $(\T^0_{\V}(t))_{t\ge 0}$ is $L^\infty$-contractive, i.e., $\|\T^0_{\V}(t)\f\|_{\infty}\le \|\f\|_{\infty}$
for every $\f \in L^2(\Om;\Cm)\cap L^\infty(\Om;\Cm)$ and every $t \ge 0$.
By the characterization of $L^\infty$-contractivity for vector-valued semigroups proved in \cite[Theorem 2.29]{Ou}, it suffices to prove that 
\begin{eqnarray*} 
\re \mathfrak a^0(\uu, \uu-(1 \wedge |\uu|){\rm sign}\hskip 1pt\uu)\geq 0
\end{eqnarray*}
for every $\uu\in \mathcal V$.
First of all, let us observe that $\uu-(1 \wedge |\uu|){\rm sign}\hskip 1pt\uu \in\V$ by Hypothesis \ref{V}(ii). Moreover, since
\begin{eqnarray*} 
\uu-(1 \wedge |\uu|){\rm sign}\hskip 1pt\uu=(|\uu|-1)^+{\rm sign}\hskip 1pt\uu=\uu\bigg (1-\frac{1}{|\uu|}\bigg )\chi_{\{|\uu|>1\}},
\end{eqnarray*}
it follows by Lemma \ref{lemma-A1} that
\begin{eqnarray*} 
\nabla (\uu-(1 \wedge |\uu|){\rm sign}\hskip 1pt\uu)= \bigg (1-\frac{1}{|\uu|}\bigg )\chi_{\{|\uu|> 1\}}\nabla \uu -\uu\chi_{\{|\uu|> 1\}}\nabla\bigg (\frac{1}{|\uu|}\bigg ).
\end{eqnarray*}
Consequently, taking into account that 
\begin{align}
\nabla |\uu|= |\uu|^{-1}\re \sum_{i=1}^m (\nabla u_i)\overline{u_i}\chi_{\{\uu\not=0\}},
\label{star1}
\end{align}
 we deduce
\begin{align*}
\re \mathfrak a^0(\uu, (|\uu|-1)^+{\rm sign}\hskip 1pt\uu)=&\int_{|\uu|>1}\big ((Q\nabla\uu, \nabla\uu)+\re(V\uu,\uu)\big )\bigg (1-\frac{1}{|\uu|}\bigg )dx\\
&-\re \int_{|\uu|>1}\sum_{i=1}^m \overline{u_i}(Q\nabla u_i, \nabla|\uu|^{-1}) dx \\
=&\int_{|\uu|>1}\big ((Q\nabla\uu, \nabla\uu)+\re(V\uu,\uu)\big )\bigg (1-\frac{1}{|\uu|}\bigg )dx\\
&+\int_{|\uu|>1}\frac{1}{|\uu|} (Q\nabla|\uu|, \nabla|\uu|)dx\geq 0, 
\end{align*}
thanks to Hypotheses \ref{base}. Thus, $(\T^0_{\V}(t))_{t\ge 0}$ is a $L^\infty$-contractive semigroup and we can extend it to a semigroup on $L^p(\Om;\C^m)$ for $p\in [2,\infty[\,\cup\{\infty\}$. Note that the form $\mathfrak a^0$ is symmetric. Hence, the same conclusion holds for $p\in [1,2[$ since the adjoint semigroup $((\T^0_{\V}(t))^*)_{t\ge 0}$ is $L^\infty$-contractive too. 
% Since both $(\T^0_{\V}(t))_{t\ge 0}$ and $((\T^0_{\V}(t))^*)_{t\ge 0}$ are $L^\infty$-contractive, then the semigroup $(\T^0_{\V}(t))_{t\ge 0}$ extends to a strongly continuous semigroup of contractions on $L^p(\Om;\C^m)$ for every  $p\in [1,\infty[$\,. 
We have so proved that the semigroup $(\T^0_{\V}(t))_{t\ge 0}$ extends to a strongly continuous semigroup of contractions on $L^p(\Om;\C^m)$ for every  $p\in [1,\infty[$\,.
\end{proof}

\section{$L^p$-quasicontractivity for semigroups associated to strongly coupled elliptic operators}
\label{sect-3}

To begin with, we prove that we can associate a semigroup of bounded operators with the form $\mathfrak{a}$  defined by
\begin{align}
\af(\uu,\vv)
=\int_{\Omega}\bigg [&\sum_{h,k=1}^d (Q^{hk}D_k \uu, D_h\vv)\!+\!\sum_{h=1}^d[( B^h D_h \uu, \vv)\!+\! ( C^h\uu, D_h\vv)]\!+\!((V\!+\!W)\uu,\vv)\bigg ]dx
\label{form}
\end{align}
for every $\uu,\vv\in \mathcal V$, where $Q^{hk}:=q_{hk} I + A^{hk}$ for every $h,k=1,\ldots,d$.

\begin{prop}
Under Hypotheses $\ref{base}$ and $\ref{V}$, we can associate a strongly continuous analytic semigroup $(\T_2(t))_{t\ge 0}$ with the form $\mathfrak{a}$ in $L^2(\Omega;\C^m)$, which satisfies the estimate
\begin{equation} 
\|\T_2(t)\|_{\mathcal{L}(L^2(\Om;\C^m))}\leq e^{C_\gamma \left(\kappa_W+\frac{(\kappa_B+\kappa_C)^2}{4}\right)t}, \qquad \;\, t\in [0,\infty[\,.
\label{leggi}
\end{equation}
\end{prop}
\begin{proof} 
Note that
$\af$ is densely defined by the assumption on $\V$. Moreover, 
observing that 
\begin{eqnarray*} 
\gamma (V_S\uu, \uu)+ C_\gamma|\uu|^2\leq (V_S\uu, \uu)(\gamma + C_\gamma v_0^{-1}),\qquad\;\, \uu \in\V,
\end{eqnarray*}
and taking \eqref{cond-c0} and \eqref{CS} into account, we get that 
\begin{align*}
|\af(\uu,\vv)|\le&\int_{\Omega}\sum_{i=1}^m (Q\nabla u_i,\nabla v_i)dx+2\kappa_A \int_{\Omega}\bigg (\sum_{i=1}^m (Q\nabla u_i,\nabla u_i)\bigg )^{\frac{1}{2}}\bigg (\sum_{i=1}^m (Q\nabla v_i,\nabla v_i)\bigg )^{\frac{1}{2}}dx\\
&+\kappa_B\int_\Om\bigg (\sum_{i=1}^m (Q\nabla u_i,\nabla u_i)\bigg )^{\frac{1}{2}}(\gamma (V_S\vv,\vv) + C_\gamma |\vv|^2)^{\frac{1}{2}}dx\\
&+\kappa_C\int_\Om\bigg (\sum_{i=1}^m (Q\nabla v_i,\nabla v_i)\bigg )^{\frac{1}{2}}(\gamma(V_S\uu,\uu) + C_\gamma |\uu|^2)^{\frac{1}{2}}dx\\
&+(1+c_0 +\kappa_W(\gamma+ C_\gamma v_0^{-1}))\int_\Om(V_S\uu, \uu)^{\frac{1}{2}}(V_S\vv, \vv)^{\frac{1}{2}}dx\\
\le & K_*\|\uu\|_{D_{Q,V} }\|\vv\|_{D_{Q,V}}
\end{align*}
for every $\uu,\vv \in \V$, where $K_*$ is a suitable positive constant, independent of $\uu$ and $\vv$.
Hence, $\af$ is continuous on $\V\times\V$.

Moreover, for every $\uu\in \mathcal V$ and $\varepsilon \in ]0,1[$ such that $\frac 1 \gamma - \kappa_W-\frac{(\kappa_B+\kappa_C)^2}{4\varepsilon}>0$ (which exists due to Hypothesis \ref{base}(v)), we can estimate:
\begin{align*}
{\rm Re}\hskip 1pt \af(\uu,\uu)\ge &\int_{\Om}\bigg [\sum_{i=1}^m(Q\nabla u_i,\nabla u_i)dx\\
&-\int_{\Om}( \kappa_B + \kappa_C)\bigg (\sum_{i=1}^m(Q\nabla u_i,\nabla u_i)\bigg )^{\frac{1}{2}}(\gamma(V_S\uu, \uu) + C_\gamma|\uu|^2)^{\frac{1}{2}}dx\\
&+\int_\Om \big [(V_S\uu,\uu)-\kappa_W(\gamma(V_S\uu, \uu) + C_\gamma|\uu|^2)\big]dx\\
\ge &(1-\varepsilon) \int_{\Om}\sum_{i=1}^m(Q\nabla u_i,\nabla u_i)dx + \bigg ( 1 - \gamma\kappa_W-\frac{\gamma(\kappa_B+\kappa_C)^2}{4\varepsilon}\bigg ) \int_{\Om}(V_S\uu, \uu)dx  \\
 &-C_\gamma \bigg (\kappa_W+\frac{(\kappa_B+\kappa_C)^2}{4\varepsilon}\bigg )\|\uu\|_2^2\\
\ge & \min\bigg\{1-\varepsilon, 1 - \gamma\kappa_W-\gamma\frac{(\kappa_B+\kappa_C)^2}{4\varepsilon}\bigg\}\|\uu\|_{D_{Q,V}}^2\\
&-C_\gamma \bigg(\kappa_W+\frac{(\kappa_B+\kappa_C)^2}{4\varepsilon}\bigg)\|\uu\|_2^2
\end{align*}
for every $\uu\in\V$. We have so proved that $\af$ is elliptic. As a byproduct, the opposite of the operator $\A_\V$ associated with $\af$ generates an analytic strongly continuous  semigroup $(\T_2(t))_{t\ge0}$ on $L^2(\Omega; \C^m)$. Letting $\varepsilon$ tend to $1$ from the left, we get that 
\begin{eqnarray*} 
\re\af(\uu,\uu) 
\geq -C_\gamma \bigg (\kappa_W+\frac{(\kappa_B+\kappa_C)^2}{4}\bigg )\|\uu\|_2^2,
\end{eqnarray*}
whence the claim.
\end{proof}

% \marginpar{questa stima mi sembra naturale per il semigruppo}
% Since 
% \begin{align*}
% {\rm Re}\hskip 1pt \af(\uu,\uu)\ge &\int_{\Om} (Q\nabla \uu,\nabla \uu) -\left( \kappa_B + \kappa_C \right)(Q\nabla \uu,\nabla \uu)^{\frac{1}{2}}\left(\gamma{\rm Re}(V\uu, \uu) + C_\gamma|\uu|^2\right)^{\frac{1}{2}}dx\\
% &+\int_\Om {\rm Re}(V\uu,\uu)-\kappa_W\left(\gamma{\rm Re}(V\uu, \uu) + C_\gamma|\uu|^2\right)dx\\
% =&\int_{\Om} (Q\nabla \uu,\nabla \uu) -\left( \kappa_B + \kappa_C \right)(Q\nabla \uu,\nabla \uu)^{\frac{1}{2}}\left(\gamma{\rm Re}(V\uu, \uu) + C_\gamma|\uu|^2\right)^{\frac{1}{2}}dx\\
% & + \int_{\Omega}\bigg[\left(\frac 1 \gamma -\kappa_W\right)(\gamma{\rm Re}(V\uu,\uu) +C_\gamma|\uu|^2)- \frac{C_\gamma}{\gamma}|\uu|^2\bigg] dx\\
% \ge & \left(\frac 1 \gamma -\kappa_W-\frac{(\kappa_B + \kappa_C )^2}{4}\right)\int_\Om (\gamma{\rm Re}(V\uu,\uu) +C_\gamma|\uu|^2)dx- \frac{C_\gamma}{\gamma}\|\uu\|^2_2\\
% \ge &-\frac{C_\gamma}{\gamma}\|\uu\|^2_2,
% \end{align*}
% we get that
% \begin{eqnarray*} 
% \|\T_2(t)\|_{2\to 2}\leq e^{\frac{C_\gamma}{\gamma} t}, \qquad\;\, t>0.
% \end{eqnarray*}

Distributionally, the operator $\A_\V$ associated with $\af$ acts  as follows:
\begin{equation*}
\A_\V\uu=- \sum_{h,k=1}^d D_h(Q^{hk}D_k \uu)+\sum_{h=1}^d  B^h D_h \uu-\sum_{h=1}^d D_h( C^h\uu)+(V+W)\uu.
\end{equation*}

\begin{rmk}{\rm The adjoint form $\af^*:\V\times \V \to \C$ is defined as follows:
\begin{align}
\af^*(\uu,\vv)=&\int_{\Omega}\bigg (\sum_{h,k=1}^d ((Q^{kh})^TD_k \uu, D_h\vv)+\sum_{h=1}^d\big [ ((C^h)^TD_h \uu, \vv)+ ((B^h)^T\uu, D_h\vv)\big  ]\bigg )dx\notag\\
&+\int_{\Omega}((V+W)^T\uu,\vv)dx\label{adjoint}
\end{align}
for every $\uu,\vv \in \V$. By Remark \ref{rem-2.2}(ii),
% $\af^*$ is a continuous and elliptic form. Moreover, since $(M^T \xi, \eta)=\overline{(M\eta,\xi)}$ for every $\R^{m\times m}$-matrix $M$ and every $\xi, \eta \in \C^m$ and, consequently, $|{\rm Im}(M^T \xi, \eta)|=|{\rm Im}(M\eta,\xi)|$ and ${\rm Re}(M^T \xi, \eta)={\rm Re}(M\eta,\xi)$, the transposed matrices $(A^{kh})^T$, $(B^h)^T$, $(C^h)^T$ and $W^T$ satisfy the same assumptions as $A^{hk}$, $C^h$, $B^h$ and $W$ respectively. Consequently, the operator 
$-\A_\V^*$ 
%associated with $\af^*$ (which is the adjoint operator of $\A_\V$) 
generates an analytic and strongly continuous semigroup $(\T_2^*(t))_{t \ge 0}$  (which is the adjoint semigroup  of $(\T_2(t))_{t \ge 0}$) in $L^2(\Om;\C^m)$ satisfying  estimate \eqref{leggi}.
Distributionally, $\A_\V^*$ is given by the expression:

\begin{equation*}
\A_\V^*\uu=- \sum_{h,k=1}^d D_h((Q^{kh})^TD_k \uu)+\sum_{h=1}^d  (C^h)^T D_h \uu-\sum_{h=1}^d D_h( (B^h)^T\uu)+(V+W)^T\uu.
\end{equation*}
}
\end{rmk}

The following theorem proves that the semigroup $(\T_2(t))_{t\ge0}$ can be extended to a bounded semigroup in $L^p(\Omega;\Cm)$ for $p$ belonging to some intervals $\mathcal{I}$. For notational convenience, we set
\begin{align*}
\delta_1=\frac{K}{(\kappa_A^2+1)K+(\kappa_A(\kappa_B+\kappa_C)+\kappa_B)^2},\qquad\;\, 
\delta_2=\frac{K}{\kappa_A^2K+(\kappa_A(\kappa_B+\kappa_C)+\kappa_C)^2},
\end{align*}
where the constants $\kappa_A$, $\kappa_B$, $\kappa_C$ and $K$ are defined in Hypotheses \ref{base}.

\begin{thm}\label{main} Under Hypotheses $\ref{base}$ and $\ref{V}$, the semigroup $(\T_2(t))_{t\geq 0}$ extends consistently to  a semigroup $(\T_p(t))_{t\geq 0}$ in $L^p(\Omega;\Cm)$ for any $p\in \mathcal{I}$, where
\begin{eqnarray*} 
\mathcal{I}=\begin{cases}
]1,\infty[\,,\qquad & \kappa_A=\kappa_B=\kappa_C=0,\\[2mm]
\Big [1+\frac{\gamma\kappa_B^2}{4\left(1-\gamma\kappa_W\right)},\infty\Big [\,,\qquad &\kappa_A=\kappa_C=0,\,\kappa_B\neq 0,\\[2mm]
\Big ]1,1+\frac{4(1-\gamma k_W)}{\gamma\kappa_C^2}\Big ],\qquad &\kappa_A=\kappa_B=0,\, \kappa_C\neq 0,\\[2mm]
[2-\delta_1,2+ \delta_2],
&\kappa_A \neq 0\vee (\kappa_A=0\wedge \kappa_B\neq 0\wedge \kappa_C \neq 0).
\end{cases}
\end{eqnarray*}
Moreover, for every $\f\in L^p(\Om;\Cm)$ and $p\in\mathcal I$ it holds that
\begin{equation}\label{est_norm_gamma}
\|\T_p(t)\f\|_p\leq e^{\frac{C_\gamma}{\gamma}t}\|\f\|_p, \qquad\;\, t\in [0,\infty[\,.
\end{equation}
\end{thm}

\begin{proof} We will prove the claim by applying Nittka's criterion. Since by Proposition \ref{pcontr}, $P_p(\V)\subset \V$ for every $p\in ]1,\infty]$\,, it suffices to prove that ${\rm Re}\, \af(\uu,|\uu|^{p-1}{\rm sign}\hskip 1pt\uu)\ge 0$ for every $\uu\in \V$ such that $|\uu|^{p-1}{\rm sign}\hskip 1pt\uu \in \V$. To this aim, we fix $\uu$ as above. Recalling that 
\begin{eqnarray*} 
D_k(|\uu|^{p-1}{\rm sign}\hskip 1pt\uu\big) =|\uu|^{p-2}(D_k \uu + (p-2)({\rm sign}\hskip 1pt\uu) D_k|\uu|)
\end{eqnarray*}
(see e.g., the proof of \cite[Theorem 3.1]{ATS}) and taking \eqref{star1} into account, we deduce 
\begin{align*}
&\re\af(\uu, |\uu|^{p-1}{\rm sign}\uu)\\
=&\int_{\Om}\bigg[\sum_{i=1}^m (Q \nabla u_i, \nabla u_i)|\uu|^{p-2}+(p-2)(Q\nabla |\uu|,\nabla|\uu|)|\uu|^{p-2} \\
&\qquad + \re\sum_{h,k=1}^d (A^{hk}D_k \uu, D_h\uu) |\uu|^{p-2}
 + (p-2)\re \sum_{h,k=1}^d (A^{hk}D_k \uu, \uu)|\uu|^{p-3}D_h|\uu| \\
&\qquad+
\re \sum_{h=1}^d[(B^h D_h \uu,\uu)+(C^h\uu,D_h\uu)]|\uu|^{p-2}\\
&\qquad+ (p-2)
\re \sum_{h=1}^d (C^h\uu,\uu |\uu|^{p-3})D_h |\uu|+\re((V+W)\uu,\uu) 
|\uu|^{p-2}\bigg] dx.
\end{align*}
\noindent

We begin by considering the case $p \in ]2,\infty[$ and use Hypotheses \ref{base}(iii), see also \eqref{CS}, to estimate 
\begin{eqnarray*} 
\re\sum_{h,k=1}^d (A^{hk}D_k \uu, D_h\uu) \ge 0,
\end{eqnarray*}
\begin{align}
\re \sum_{h,k=1}^d (A^{hk}D_k \uu, \uu)D_h|\uu|&\ge - \bigg |\sum_{h,k=1}^d (A^{hk}D_k \uu, \uu D_h|\uu|)\bigg|\notag\\
&\ge - 2\kappa_A |\uu|\left(Q \nabla |\uu|,\nabla |\uu|\right)^{\frac{1}{2}}\bigg (\sum_{i=1}^m(Q \nabla u_i, \nabla u_i)\bigg )^{\frac{1}{2}}.
\label{pedalo}
\end{align}
Further,
\begin{align*}
&\re\sum_{h=1}^d[ (B^h D_h \uu,\uu)+(C^h\uu, D_h \uu)]\\
\ge &-\bigg |\sum_{h=1}^d(B^h D_h\uu,\uu)\bigg |-\bigg |\sum_{h=1}^d(C^h\uu,D_h\uu)\bigg |\\
\ge &-(\kappa_B+\kappa_C)\bigg (\sum_{i=1}^m(Q \nabla u_i,\nabla u_i)\bigg )^{\frac{1}{2}}(\gamma (V_S\uu,\uu)+C_\gamma|\uu|^2)^{\frac{1}{2}}.
\end{align*}
Analogously, 
\begin{align}\label{num}\re \sum_{h=1}^d (C^h\uu,\uu D_h |\uu|)
% &\ge-\bigg| \sum_{h=1}^d(C^h \uu, \uu D_h|\uu|)\bigg|\notag\\
&\ge -\kappa_C|\uu|(Q \nabla |\uu|,\nabla|\uu|)^{\frac{1}{2}}(\gamma (V_S\uu,\uu)+C_\gamma|\uu|^2)^{\frac{1}{2}}.
\end{align}
Finally, 
\begin{align*}
\re((V+W) \uu,\uu)\ge \bigg (\frac{1}{\gamma}-\kappa_W\bigg )(\gamma (V_S\uu,\uu)+C_\gamma|\uu|^2)-\frac{C_\gamma}{\gamma}|\uu|^2.
\end{align*}
Thus, putting all together and setting $X:=\left (\sum_{i=1}^m(Q \nabla u_i, \nabla u_i)\right )^{\frac{1}{2}}$, $Y:=(Q\nabla |\uu|,\nabla|\uu|)^{\frac{1}{2}}$ and $Z:= (\gamma (V_S\uu,\uu) + C_\gamma|\uu|^2)^{\frac{1}{2}}$, we get

\begin{align*}
&\re\af(\uu, |\uu|^{p-1}{\rm sign}\hskip 1pt\uu)+ \frac{C_\gamma}{\gamma}\|\uu\|_p^p \geq \int_{\Om}|\uu|^{p-2}\mathcal E(X,Y,Z)dx,
\end{align*}
where $\mathcal E$ is the quadratic form defined by
\begin{eqnarray*} 
\mathcal E(x, y, z)=x^2 + (p-2)y^2+(\gamma^{-1}-\kappa_W)z^2 - 2(p-2)\kappa_A xy - \big(\kappa_B+\kappa_C \big)xz - \kappa_C(p-2)yz 
\end{eqnarray*}
for every $x,y,z\in\R$.
By considering the associated matrix 
\begin{eqnarray*} 
M_\gamma=
\begin{pmatrix} 
1 & -(p-2)\kappa_A & -\frac{\kappa_B+\kappa_C}{2}\\[1mm]
-(p-2)\kappa_A & p-2 & -\frac{p-2}{2}\kappa_C\\[1mm]
-\frac{\kappa_B+\kappa_C}{2} & - \frac{p-2}{2}\kappa_C\ &\frac{1}{\gamma} - \kappa_W\end{pmatrix},
\end{eqnarray*}
we conclude that $\mathcal E$ is semidefinite positive if and only if $p-2\leq\kappa_A^{-2}$, when $\kappa_A \neq 0$, and the determinant of $M_{\gamma}$ is nonnegative, i.e., 
\begin{align}
K -(p-2) \big(\kappa_C^2 + 4(\gamma^{-1}-\kappa_W)\kappa_A^2 + 2\kappa_A\kappa_C(\kappa_B+\kappa_C)\big)\geq 0,
\label{cond-per-nittka}
\end{align}
(see Hypothesis \ref{base}(v)), or equivalently if $p \in\mathcal{I}_1(\kappa_A,\kappa_B,\kappa_C)$ where
\begin{align} 
\mathcal{I}_1(\kappa_A,\kappa_B,\kappa_C)=\begin{cases}
[2,\infty[\,, \qquad & (\kappa_A, \kappa_C)=(0,0),\\[1mm]
[2,2+ \delta_2],\qquad &(\kappa_A,\kappa_C)\neq (0,0).
\end{cases}
\label{J1}
\end{align}
Note that, in the case $\kappa_A\neq 0$, \eqref{cond-per-nittka} is stronger than the condition $p-2 \le \kappa_A^{-2}$. Thus, taking into account that
\begin{eqnarray*} 
\re\af(\uu, |\uu|^{p-1}{\rm sign}\hskip 1pt\uu)+ \frac{C_\gamma}{\gamma}\|\uu\|_p^p \geq 0
\end{eqnarray*}
for $p\in \mathcal{I}_1(\kappa_A,\kappa_B,\kappa_C) $ and
applying Nittka's criterion (taking Remark \ref{rem-2.4} into account) we deduce that
$(\T_2(t))_{t \ge 0}$ admits a continuous extension $(\T_p(t))_{t \ge 0}$ in $L^p(\Om; \C^m)$ for $p \in \mathcal{I}_1(\kappa_A,\kappa_B, \kappa_C)$ and  for such $p$'s estimate \eqref{est_norm_gamma}
holds true in $L^p(\Om;\Cm)$.

The case $p\in\, ]1,2[$ can be obtained by duality considering the adjoint form $\af^*$ (see \eqref{adjoint}).
In this case, since $(M^T \xi, \eta)=\overline{(M\eta,\xi)}$ for every $\R^{m\times m}$-matrix $M$ and every $\xi, \eta \in \C^m$ and, consequently, ${\rm Re}(M^T \xi, \eta)={\rm Re}(M\eta,\xi)$ and ${\rm Im}(M^T \xi, \eta)=-{\rm Im}(M\eta,\xi)$, the transposed matrices $(A^{kh})^T$, $(B^h)^T$, $(C^h)^T$ and $W^T$ satisfy the same assumptions as $A^{hk}$, $C^h$, $B^h$ and $W$ respectively. Hence, arguing similarly, we deduce that each operator $\T_2^*(t)$ admits a continuous extension $\T_{p'}^*(t)$ in $L^{p'}(\Om;\C^m)$ which satisfies estimate \eqref{est_norm} for every $p'\in J_1(\kappa_A,\kappa_C,\kappa_B)$. Coming back to $L^p(\Om;\Cm)$ and to the adjoint semigroups and using that $p'$ is the conjugate exponent of $p$ we deduce that $(\T_2(t))_{t\ge 0}$ can be extended to a continuous semigroup $(\T_p(t))_{t\ge 0}$ for every $p \in \mathcal{I}_2(\kappa_A,\kappa_B,\kappa_C)$, where
\begin{equation} 
\mathcal{I}_2(\kappa_A,\kappa_B,\kappa_C)=
\begin{cases}
]1,2], \qquad & (\kappa_A, \kappa_B)=(0,0),\\
[2-\delta_1,2],\qquad &(\kappa_A,\kappa_B)\neq (0,0).
\end{cases}
\label{J2}
\end{equation}
From \eqref{J1} and \eqref{J2}, the claim follows. 
\end{proof}

% \begin{rmk}{\rm Note that, under the assumptions of Theorem \ref{main}, if $A^{hk}=0$ for every $h,k=1, \ldots,d$, then $\kappa_A=0$ and the semigroup $(\T_2(t))_{t \ge 0}$ extends to a consistent semigroup $(\T_p(t))_{t \ge 0}$ for every $p \in \tilde{ \mathcal I}$, where
% \begin{equation*}
% \tilde {\mathcal I}=
% \left\{
% \begin{array}{ll}
% ]1,\infty[, & \kappa_B=0\wedge \kappa_C=0,\\[2mm]
% \left [1+\frac{\kappa_B^2}{4(\gamma^{-1}-\kappa_W)},\infty\right [\,,\qquad\;\, & \kappa_B\neq 0\wedge \kappa_C=0,\\[2mm]
% \left]1, 1+ \frac{4(1-\gamma \kappa_W)}{\gamma\kappa_C^2}\right],\qquad\;\, &\kappa_B= 0 \, \wedge\,\kappa_C\neq 0,\\[3mm]
% \left[1+\frac{\kappa_B^2}{K+\kappa_B^2}, 2+ \frac{K}{\kappa_C^2} \right],\qquad\;\, &\kappa_B\neq 0 \, \wedge \,\kappa_C \neq 0.
% \end{array}\right.
% \end{equation*}}
% \end{rmk}

Note that in Theorem \ref{main} we proved that the operator norm $\|\T_p(t)\|_{\mathcal L(L^p(\Om;\Cm))}$ can be estimated by an exponential term, that is independent of $p$.
%, when $p$ belongs to some interval depending on the coefficients of the operator $\A_\V$. 
A modification of Hypothesis \ref{base}(iv) allows us to prove a different estimate for the operator norm of the semigroup $(\T_p(t))_{t\ge 0}$ in $L^p(\Om;\C^m)$, which better highlights the dependence of the exponential term on $p$. 

\begin{hyp}\label{5bis} the matrices  $B^h$, $C^h$ $(h=1,\dots,d)$ and $W$ belong to $L^\infty_{\rm loc}(\Om;\R^{m\times m})$ and
there exist positive constants $\kappa_B$, $\kappa_C$, $\kappa_W$ and a function $\varphi:]0,\infty[\ra ]0,\infty[$ such that  
\begin{align}\label{est_11}&\bigg|\sum_{h=1}^d (B^h \theta^h, \eta) \bigg| \leq \kappa_B\mathcal{Q}(\theta^1,\ldots,\theta^d)^\frac{1}{2}(\gamma  (\Vs\eta, \eta) + \varphi(\gamma) |\eta|^2)^{\frac{1}{2}},
\end{align}
\begin{align}\label{est_22}
&\bigg|\sum_{h=1}^d (C^h \eta,\theta^h) \bigg| \leq  \kappa_C\mathcal{Q}(\theta^1,\ldots,\theta^d)^\frac{1}{2}(\gamma(\Vs\eta, \eta) + \varphi(\gamma) |\eta|^2)^{\frac{1}{2}},
\end{align}
\begin{align*}
&|(W\xi, \eta) | \leq \kappa_W(\gamma (\Vs\xi, \xi) + \varphi(\gamma) |\xi|^2)^\frac 1 2(\gamma (V_S\eta, \eta) + \varphi(\gamma)|\eta|^2)^\frac{1}{2}
\end{align*}
almost everywhere in $\Om$, for every $\gamma>0$ and  every $\theta^1,\ldots,\theta^d, \xi, \eta\in \Cm$.
\end{hyp}

\begin{thm}\label{main-bis} Under Hypotheses $\ref{base}(i)$-$(iii)$, $\ref{V}$ and $\ref{5bis}$, the semigroup $(\T_2(t))_{t\geq 0}$ extends consistently to  semigroups $(\T_p(t))_{t\geq 0}$ in $L^p(\Omega;\Cm)$ for every $p\in ]1,\infty[$\,, if $\kappa_A=0$, and for every $p\in \left ]1+\frac{2\kappa_A}{2\kappa_A+1},2+\frac{1}{2\kappa_A}\right [$\,, otherwise. Moreover, 
\begin{equation}
\label{est_norm}
\|\T_p(t)\f\|_p\leq e^{\gamma_{p}^{-1}\varphi(\gamma_p)t}\|\f\|_p, \qquad\;\, t\in [0,\infty[\,,
\end{equation}
for every $\f\in L^p(\Omega;\C^m)$, where 
\begin{eqnarray*}
\gamma_p=\left(\kappa_W+\frac{(\kappa_B+(p-1)\kappa_C)^2}{4[((p-1)^2\wedge 1)-2\kappa_A((p-1)\wedge 1)|p-2|]}\right)^{-1}.
\end{eqnarray*}

% \begin{eqnarray*} 
% c_p=C_{\left(\kappa_W+\frac{(\kappa_B+(p-1)\kappa_C)^2}{4(1-2\kappa_A(p-2))}\right)^{-1}}\left (\kappa_W+\frac{(\kappa_B+(p-1)\kappa_C)^2}{4(1-2\kappa_A(p-2))}\right ).
%\end{eqnarray*}
% if $p\in [2,\infty[$ and
% \begin{eqnarray*} 
%c_p=C_{\left(\kappa_W+\frac{[\kappa_B+(p-1)\kappa_C]^2}{4(p-1)[p-1-2\kappa_A(2-p)]}\right)^{-1}}\left(\kappa_W+\frac{[\kappa_B+(p-1)\kappa_C]^2}{4(p-1)[p-1-2\kappa_A(2-p)]}\right)
% \end{eqnarray*}
%if $p\in ]1,2[$\,.
\end{thm}

\begin{proof} The proof follows the same lines as that of Theorem \ref{main}. Hence, we limit ourselves to emphasize the main differences. Clearly, it suffices to consider the case $p\in\, ]2,\infty[$\,, since the case $p=2$ follows by estimate \eqref{leggi} (with $C_{\gamma}$ being replaced by $\varphi(\gamma_2)$) and the case $p\in\, ]1,2[$ follows by duality as in the proof of Theorem \ref{main}. For this purpose, we fix $p\in\, ]2,\infty[$ and $\uu\in\V$ such that $|\uu|^{p-1}{\rm sign}\,\uu \in \V$. Observing that
\begin{equation}
\label{dis-solita}(Q\nabla|\uu|, \nabla|\uu|)\leq \sum_{i=1}^m (Q\nabla u_i, \nabla u_i),
\end{equation}
from \eqref{pedalo} we can estimate
\begin{align*}
\re \sum_{h,k=1}^d (A^{hk}D_k \uu, \uu)D_h|\uu|
&\ge - 2\kappa_A |\uu|\sum_{i=1}^m(Q \nabla u_i,\nabla u_i).
\end{align*}
Estimating the other terms in the same way as in the proof of the quoted theorem, and using again \eqref{dis-solita} in \eqref{num}, we conclude that
\begin{align*}
\re\af(\uu, |\uu|^{p-1}{\rm sign}\hskip 1pt\uu) \ge & \int_{\Om}\mathcal E_\gamma(X,Y)|\uu|^{p-2}dx
-\frac{\varphi(\gamma)}\gamma \int_\Omega |\uu|^pdx\end{align*}
where $X:=(Q\nabla\uu,\nabla\uu)^{\frac{1}{2}}$, $Y:=(\gamma (V_S\uu,\uu)+\varphi(\gamma)|\uu|^2)^{\frac{1}{2}}$, and
$\mathcal E_\gamma$ is the quadratic form defined by
\begin{eqnarray*} 
\mathcal E_\gamma(x, y)=(1+2\kappa_A(2-p))x^2-(\kappa_B+(p-1)\kappa_C)xy+(\gamma^{-1}-\kappa_W)y^2,\qquad\;\,x,y\in\R.
\end{eqnarray*}
The form $\mathcal E_\gamma$ is semidefinite positive if $1+2k_A(2-p)>0$ and 
\begin{eqnarray*} 
\gamma\le \gamma_p:= \left (\kappa_W+\frac{(\kappa_B+(p-1)\kappa_C)^2}{4(1+2\kappa_A(2-p))}\right )^{-1}.
\end{eqnarray*}
From the previous estimate it follows that
\begin{align*}
\re\af(\uu, |\uu|^{p-1}{\rm sign}\hskip 1pt\uu) \ge 
-\frac{\varphi(\gamma_p)}{\gamma_p}\int_{\Omega}|\uu|^pdx,
\end{align*}
so that, using Proposition \ref{pcontr} and applying Nittka's criterion, we conclude that $(\T_2(t))_{t\ge 0}$ extends consistently to a semigroup $(\T_p(t))_{t\geq 0}$ in $L^p(\Omega;\Cm)$ satisfying estimate \eqref{est_norm} for every $p\in \left [2,2+\frac{1}{2\kappa_A}\right [$\,, if $\kappa_A\neq 0$ and for every $p\in [2,\infty[$ otherwise.
\end{proof}

\begin{rmk} 
{\rm As already pointed out, the domains of the generators of the semigroups $(\T_p)_{t\geq 0}$  are described only in an abstract way through the choice of the space $\V$. Nevertheless, under additional regularity assumptions on the coefficients, an explicit description is possible. For example, if $\Omega =\Rd$ and $\V=\overline{C_c^\infty(\Rd;\Cm)}$, by using conditions that ensure that $C_c^\infty(\Rd; \Cm)$ is a core for the generators of the semigroups, the description of the domain has been obtained in \cite[Section 4]{AngLorMan}  for  operators  possibly coupled up to the second-order (i.e., $A^{hk}\not= 0$) and in \cite{ALMR} in the case of operators coupled up to the first-order.}
\end{rmk}

\section{Upper Gaussian estimates}
\label{sect-4}
In  this section, we assume stronger assumptions on the domain $\Omega$ and on the coefficients of the elliptic operator $\A_\V$, %defined on smooth functions $\uu:\Omega\to\Cm$  by
% \begin{equation*}
% \A_\V\uu=- \sum_{h,k=1}^d D_h(\textcolor{red}{q_{hk}}D_k \uu)+\sum_{h=1}^d  B^h D_h \uu-\sum_{h=1}^d D_h( C^h\uu)+V\uu,
% \end{equation*}
in order to show that its associated semigroup $(\T_2(t))_{t \ge 0}$ is defined by a kernel that satisfies a Gaussian estimate expressed in terms of a distance that takes into account the growth of the coefficients of the operator $\A_{\V}$.
%To this aim, we consider the following additional conditions on the coefficients of the operator $\A_{\V}$ and on $\Omega$.

\begin{hyp}\label{base-2}
\begin{enumerate}[\rm (i)]
\item 
$\Omega$ has the extension property;
\item
there exists a positive constant $\nu_0$ such that 
$\lambda_Q(x)\ge\nu_0$ for almost every $x\in\Omega$;
\item 
 $A^{hk}$ $(h,k=1,\ldots,d)$ and $W$ identically vanish on $\Omega$;
\item 
%$A_{hk}=0$ for every $h,k=1,\ldots,d$ and  
there exist  positive constants $\beta$, $\kappa$ and $c\ge 1$ such that, for every $\gamma>0$, every $\theta^1,\ldots,\theta^d, \xi, \eta\in \Cm$ and almost every $x\in\Om$, the following estimates hold:
\begin{align*} 
&\bigg|\sum_{h=1}^d (B^h(x) \theta^h, \eta) \bigg| \leq \kappa\big (\mathcal{Q}(\theta^1,\ldots,\theta^d)\big )^\frac{1}{2}\big (\gamma (V_S(x)\eta, \eta) + c\gamma^{-\beta}|\eta|^2\big )^{\frac{1}{2}},\\
&\bigg|\sum_{h=1}^d (C^h(x)\eta,\theta^h) \bigg| \leq  \kappa(\mathcal{Q}(\theta^1,\ldots,\theta^d) )^\frac{1}{2} \big (\gamma (V_S(x)\eta, \eta) + c\gamma^{-\beta} |\eta|^2\big )^{\frac{1}{2}}.
\end{align*}
\end{enumerate}
\end{hyp}

We prove some kernel estimates expressed in term of the distance $d_{Q,V,\beta}$ defined by 
\begin{eqnarray}\label{distanza}
d_{Q,V,\beta}(x,y)=\sup\Big\{|\psi(x)-\psi(y)|:\psi\in C_c^\infty(\Om)
, |(Q \nabla\psi, \nabla \psi)|  \leq \lambda_V^{\frac{\beta}{\beta+1}}\,  a.e.\, in \,\,\Om\Big\},
\end{eqnarray}
for every $x,y\in\Om$.

\begin{rmk}{\rm
\begin{enumerate}[\rm (i)]
\item 
We emphasize that, if $Q$ and $V$ are bounded, then $d_{Q,V,\beta}$ is equivalent to the Euclidean distance (see e.g., \cite[Corollary 6.15]{Ou}).
In the general case, $d_{Q,V,\beta}$ is equivalent to the Euclidean norm if there exist two positive constants  $q_0$ and $q_1$ such that 
$q_0 \lambda_V(x)^{\frac{\beta}{\beta+1}}|\xi|^2  \leq (Q(x)\xi, \xi) \leq q_1\lambda_V(x)^{\frac{\beta}{\beta+1}}|\xi|^2$ for every $\xi\in\Rd$ and almost every $x\in \Om$ (see e.g., \cite[Theorem 7]{daviesexplicit}).
\item 
Hypotheses $\ref{base}(i)$-$(ii)$, $\ref{V}$ and $\ref{base-2}$ allow us to apply Theorem \ref{main-bis} and extend consistently the semigroup $(\T_2(t))_{t\ge 0}$ to the space $L^p(\Om;\C^m)$ for every $p \in ]1,\infty[$. For this reason, in the following theorem, we simply write $\T(t)$ instead of $\T_p(t)$ for every $t\in [0,\infty[$\,.
\end{enumerate}
}\end{rmk}

\begin{thm}\label{Gauss}
Assume that Hypotheses $\ref{base}(i)$-$(ii)$ and $\ref{base-2}$ are satisfied.
Then, the semigroup  $(\T(t))_{t \ge 0}$ is bounded from $L^1(\Omega; \C^m)$ into $L^{\infty}(\Omega;\C^m)$. Hence, it is described by a kernel $\bm{k}(t,\cdot, \cdot) \in L^\infty(\Omega\times \Omega; \C^{m\times m})$, which, in addition, satisfies the estimate
% \begin{align}
% |k_{ij}(t,x,y)|\leq 
% C_{d,\beta}^2&\nu_0^{-\frac{d}{2}}e^2
% \left(t^{-1}+H_1+ H_2 \left(t^{-1}d_{Q,V,\beta}(x,y)\right)^{\frac{2\beta+2}{2\beta+1}}\right )^{\frac{d}{2}}e^{2^{2\beta+2}H_3t}\notag\\
% &\times e^{-H_4t^{-\frac{1}{2\beta+1}}d_{Q,V,\beta}(x,y)^{\frac{2\beta+2}{2\beta+1}}},\label{ker-est}
% \end{align}
% for any $i,j=1, \ldots,m$, almost every $x, y \in \Omega$ and some  positive constants $H_i$, $i=1,\dots,4$ depending on $\beta, d, \kappa, c$.
\begin{align}
|k_{ij}(t,x,y)|&\leq 
C_0\left(1+t^{-1}+\left(t^{-1}d_{Q,V,\beta}(x,y)\right)^{\frac{2\beta+2}{2\beta+1}}\right )^{\frac{d}{2}}e^{C_1t-C_2t^{-\frac{1}{2\beta+1}}d_{Q,V,\beta}(x,y)^{\frac{2\beta+2}{2\beta+1}}}
\label{ker-est}
\end{align}
for every $i,j=1, \ldots,m$, for almost every $x, y \in \Omega$ and for some positive constants $C_0$, $C_1$ and $C_2$, explicitly computed in the proof, In particular, $C_0$ depends on $d,\beta,\nu_0,\kappa,c$, whereas $C_1$ and $C_2$, depend only on $\beta,\kappa,c$. 
\end{thm}

\begin{proof}
We introduce the set 
\begin{equation*}
\mathcal W=\Big\{\psi \in C^\infty_c(\Om):  (Q\nabla\psi, \nabla \psi)\leq \lambda_V^{\frac{\beta}{\beta+1}}\textit{ a.e. in $\Omega$}\Big\}
\end{equation*}
and observe that, if $\psi\in \mathcal{W}$, then  $e^{\sigma\psi}\V\subseteq \V$ for every $\sigma \in \R$.  Indeed, if $\uu\in\V$, then $(e^{\sigma\psi}-1)\uu$ belongs to $W^{1,2}_0(\Om;\C^m)$ , so it can be approximated in the $W^{1,2}(\Omega;\C^m)$-norm by a sequence $(\bm\varphi_n)$ of smooth and compactly supported functions. Since the supports of these functions can be assumed to be contained into a common compact set $K_0\subset\Omega$ and the matrix-valued functions $Q$ and $V$ are locally bounded, the sequence $(\bm\varphi_n)$ converges to $(e^{\sigma\psi}-1)\uu$ also in the $D_{Q,V}$-norm and this is enough to infer that $(e^{\sigma\psi}-1)\uu$ belongs to $\mathcal{V}$.
Thus, for every $\psi\in \mathcal W$ and $\sigma\in\R$, the form 
$\mathfrak{a}_{\sigma,\psi}:=
\mathfrak{a}(e^{-\sigma\psi}\cdot,e^{\sigma\psi}\cdot)$ is well-defined on $\mathcal V\times \mathcal V$.

We split the proof into four steps.

{\em Step 1.} Here, we prove that we can associate a strongly continuous and analytic semigroup $(\T^{\sigma,\psi}_2(t))_{t \ge 0}$ with the form $\af_{\sigma, \psi} $ in $L^2(\Om;\C^m)$ and that such a semigroup extends to a bounded semigroup in $L^p(\Om;\C^m)$ for every $p \in [2,\infty[$\,, which satisfies the estimate
\begin{equation}
\|\T^{\sigma,\psi}_p(t)\|_{{\mathcal L}(L^p(\Omega;\C^m))}
\le e^{\tau_{p,\sigma}t},\qquad\;\,t\in [0,\infty[\,,
\label{bryanadams-1}
\end{equation}
where
\begin{eqnarray*}
\tau_{p,\sigma} 
=c\bigg (\frac{4(\sigma^2+2\kappa|\sigma|)+p^2(|\sigma|+\kappa)^2}{4}\bigg )^{\beta+1}.
\end{eqnarray*}

A simple computation reveals that 
\begin{align*} 
&\af_{\sigma, \psi}(\uu,\vv)\notag\\
=&\int_\Omega\bigg[\sum_{i=1}^m(Q\nabla u_i, \nabla v_i) + \sum_{h=1}^d(B_{\sigma, \psi}^hD_h\uu, \vv) + \sum_{h=1}^d(C_{\sigma, \psi}^h\uu, D_h\vv) + ((V+W_{\sigma, \psi})\uu, \vv)\bigg]dx
\end{align*}
for every $\uu, \bm{v}\in\mathcal{V}$, where
\begin{eqnarray*} 
B^h_{\sigma, \psi}=B^h+(\sigma Q\nabla\psi)_hI,\qquad\;\,C^h_{\sigma, \psi}=C^h-(\sigma Q\nabla \psi)_hI
\end{eqnarray*}
for every $h=1, \dots, d$, and 
\begin{eqnarray*} 
W_{\sigma, \psi}=-\sigma^2 (Q\nabla\psi,\nabla\psi)I + \sigma \sum_{h=1}^d(C^h-B^h)D_h \psi.
\end{eqnarray*}

Note that
\begin{align*}
\bigg |\sum_{h=1}^d((Q\nabla\psi)_h\theta^h,\eta)\bigg |
=&\bigg |\sum_{i=1}^m(Q\nabla\psi,\theta_i)\overline\eta_i\bigg |
=\bigg |\sum_{i=1}^m(Q^{\frac{1}{2}}\nabla\psi,Q^{\frac{1}{2}}\theta_i)\overline\eta_i\bigg|
\le |Q^{\frac{1}{2}}\nabla\psi|\sum_{i=1}^m |Q^{\frac{1}{2}}\theta_i||\eta_i|\\
\le &|Q^{\frac{1}{2}}\nabla\psi||\eta|
\big (\mathcal{Q}(\theta^1,\ldots,\theta^d)\big )^{\frac{1}{2}}
\le \Big (\lambda_V^{\frac{\beta}{\beta+1}}|\eta|^2\Big )^{\frac{1}{2}}\big (\mathcal{Q}(\theta^1,\ldots,\theta^d)\big )^{\frac{1}{2}}\\
\le & \bigg(\gamma\lambda_V|\eta|^2+\frac{\beta^{\beta}}{(\beta+1)^{\beta+1}}\gamma^{-\beta}|\eta|^2\bigg )^{\frac{1}{2}}
\big (\mathcal{Q}(\theta^1,\ldots,\theta^d)\big )^{\frac{1}{2}}\\
\le & 
\bigg(\gamma (V_S\eta,\eta)+\frac{\beta^{\beta}}{(\beta+1)^{\beta+1}}\gamma^{-\beta}|\eta|^2\bigg )^{\frac{1}{2}}\big (\mathcal{Q}(\theta^1,\ldots,\theta^d)\big )^{\frac{1}{2}},
\end{align*}
where $\theta_i=(\theta^1_i,\ldots.\theta^d_i)$ for every $i=1,\ldots,m$. Hence, recalling that $c\ge 1$, we can estimate
\begin{align*}
\bigg |\sum_{h=1}^d(B^h_{\sigma,\psi}\theta^h,\eta)\bigg |
\le &
(\kappa+|\sigma|)\big (\gamma (V_S\eta,\eta)+c\gamma^{-\beta}|\eta|^2\big )^{\frac{1}{2}}
(\mathcal{Q}(\theta^1,\ldots,\theta^d))^{\frac{1}{2}},
\end{align*}
\begin{align*}
\bigg |\sum_{h=1}^d(C^h_{\sigma,\psi}\theta^h,\eta)\bigg |
\le &
(\kappa+|\sigma|)\big (\gamma (V_S\eta,\eta)+c\gamma^{-\beta}|\eta|^2\big )^{\frac{1}{2}}
(\mathcal{Q}(\theta^1,\ldots,\theta^d))^{\frac{1}{2}}.
\end{align*}
Moreover, arguing similarly, we get
\begin{align*}
|(W_{\sigma,\psi}\xi,\eta)|
\le &
(\sigma^2+2\kappa|\sigma|)\big (\gamma (V_S\xi,\xi)+c\gamma^{-\beta}|\xi|^2\big )^{\frac{1}{2}}\big (\gamma (V_S\eta,\eta)+c\gamma^{-\beta}|\eta|^2\big )^{\frac{1}{2}}.
\end{align*}

Thus, applying Theorem \ref{main-bis} with $\kappa_A=0$, $\kappa_B$, $\kappa_C$ and $\kappa_W$ being replaced by
$\kappa+|\sigma|$, $\kappa+|\sigma|$ and
$\sigma^2+2\kappa|\sigma|$, respectively, we conclude that we can associate a semigroup $(\T^{\sigma,\psi}_2(t))_{t\ge 0}$ with the form $\mathfrak{a}_{\sigma,\psi}$ in $L^2(\Om;\C^m)$ and this semigroup can be extended consistently to $L^p(\Omega;\C^m)$ for every $p\in ]1,\infty[$\,, with a strongly continuous semigroup $(\T^{\sigma,\psi}_p(t))_{t\ge 0}$, 
which satisfies estimate \eqref{bryanadams-1}. Due to the consistency of these semigroups, to ease the notation we simply write $\T^{\sigma,\psi}(t)$ instead of $\T^{\sigma,\psi}_p(t)$ for every $t\in [0,\infty[$\,. 
%\begin{equation}
%\|\T^{\sigma,\psi}_p(t)\|_{p\to p}
%\le e^{c\gamma^{-\beta-1}t},\qquad\;\,t\ge 0.
%\label{bryanadams-0}
%\end{equation}
%Recalling that $p$ and $\gamma$ are related by the condition
%\begin{eqnarray*}
%1-\gamma(\sigma^2+2\kappa|\sigma|)-\gamma\frac{p^2}{4}(|\sigma|+\kappa)^2>0,
%\end{eqnarray*}
%for every fixed $p\in [2,\infty[$, we can take
%\begin{equation}
%\gamma=\overline{\gamma}_p=\frac{2}{2(\sigma^2+2\kappa|\sigma|)+p^2(|\sigma|+\kappa)^2}.
%\label{pioggia}
%\end{equation}
%in \eqref{bryanadams-0} to get the claim.}

{\em Step 2}. Here, we prove that 
each operator $\T^{\sigma,\psi}(t)$ maps $L^2(\Omega;\C^m)$ into $L^{\frac{2r}{r-2}}(\Omega;\C^m)$ and there exists a positive constant $c_{r,d}$ such that
\begin{align}
\|\T^{\sigma,\psi}(t)\|_{\mathcal{L}(L^2(\Omega;\C^m);L^{\frac{2r}{r-2}}(\Omega;\C^m))}\le
 c_{r,d}\nu_0^{-\frac{d}{2r}}t^{-\frac{d}{2r}}
e^{\big (\hat\tau_{\frac{2r}{r-2},\sigma}+\frac{\nu_0}{2}\big ) t},\qquad\;\, t\in ]0,\infty[\,,
\label{pitbike-2}
\end{align}
where  $r=d$, if $d\ge 3$, and $r$ is any number larger than $2$, if $d=1,2$, and
\begin{equation}\label{hatau}
\hat\tau_{p,\sigma}=c\bigg (\sigma^2+2\kappa|\sigma|+\frac{p^2}{2}(|\sigma|+\kappa)^2\bigg )^{\beta+1}
\end{equation}
for every $p \in [2,\infty[$ and $\sigma \in\R$.

In what follows, sometimes,  we simply write $r^*$ instead of $\frac{2r}{r-2}$ to ease the notation.

We begin by observing that, for every $\uu\in \V$, we can estimate
\begin{align*}
&{\rm Re}\hskip 1pt \af_{\sigma,\psi}(\uu,\uu)\\
\ge &\int_{\Om}\bigg[\sum_{i=1}^m(Q\nabla u_i,\nabla u_i) -2(\kappa+|\sigma|)\bigg (\sum_{i=1}^m(Q\nabla u_i,\nabla u_i)\bigg )^{\frac{1}{2}}\big (\gamma(V_S\uu, \uu) + c\gamma^{-\beta}|\uu|^2\big )^{\frac{1}{2}}\bigg ]dx\\
&+\int_\Om \big [(V_S\uu,\uu)-(\sigma^2+2\kappa|\sigma|)\big (\gamma(V_S\uu, \uu) + c\gamma^{-\beta}|\uu|^2\big )\big ]dx\\
\ge &(1-\varepsilon) \int_{\Om}\sum_{i=1}^m(Q\nabla u_i,\nabla u_i)dx
+ \left( 1 - \gamma(\sigma^2+2\kappa|\sigma|)-\frac{\gamma(\kappa+|\sigma|)^2}{\varepsilon}\right) \int_{\Om}(V_S\uu, \uu)dx  \\
 &-c\gamma ^{-\beta}\bigg (\sigma^2+2\kappa|\sigma|+\frac{(\kappa+|\sigma|)^2}{\varepsilon}\bigg )\|\uu\|_2^2.
\end{align*}
for every $\varepsilon, \gamma>0$.
If we take $\varepsilon=\frac{1}{2}$
and $\gamma=(\sigma^2+2\kappa|\sigma|+2(|\sigma|+\kappa)^2)^{-1}$, then the second term in the last side of the previous chain of inequalities vanishes, and taking into account the definition of $\hat \tau_{2,\sigma}$ (see \eqref{hatau}), we obtain that
\begin{align}
{\rm Re}\hskip 1pt \af_{\sigma,\psi}(\uu,\uu)\ge 
&\frac{1}{2} \int_{\Om}\sum_{i=1}^m (Q\nabla u_i,\nabla u_i)dx
-\hat \tau_{2,\sigma}\|\uu\|_2^2\ge \frac{\nu_0}{2} \|\nabla \uu\|_2^2-\hat\tau_{2,\sigma}\|\uu\|_2^2.
\label{Sob}
\end{align}

Since $\mathcal{V}\hookrightarrow W^{1,2}(\Omega;\C^m)$, from the Sobolev embedding theorem it follows that
\begin{eqnarray*}
\|\uu\|_{\frac{2r}{r-2}}^2\leq \hat c_{r,d}^2 \|\uu\|_2^{2-\frac{2d}{r}}(\|\uu\|_2^2+\|\nabla \uu\|_2^2)^{\frac{d}{r}},
\end{eqnarray*}
i.e.,
\begin{eqnarray*}
\|\nabla\uu\|_2^2\|\uu\|_2^{\frac{2r}{d}-2}\ge
\hat c_{r,d}^{-\frac{2r}{d}}\|\uu\|_{\frac{2r}{r-2}}^{\frac{2r}{d}}-\|\uu\|_2^{\frac{2r}{d}},
\end{eqnarray*}
for every $\uu\in\mathcal{V}$ and for some positive constant $\hat c_{r,d}$. From this inequality and 
\eqref{Sob}, we conclude that
\begin{align}
\|\uu\|_2^{\frac{2r}{d}-2}{\rm Re}\hskip 1pt \af_{\sigma,\psi}(\uu,\uu)\ge 
&\frac{\nu_0}{2}\hat{c}_{r,d}^{-\frac{2r}{d}}\|\uu\|_{r^*}^{\frac{2r}{d}}- \bigg (\hat\tau_{2,\sigma}+\frac{\nu_0}{2}\bigg )\|\uu\|_2^{\frac{2r}{d}}
,\qquad\;\,\uu\in\mathcal{V}.
\label{Sob-1}
\end{align}

Now, we observe that the function
$t\mapsto\|e^{-\left (\hat\tau_{r^*,\sigma}+\frac{\nu_0}{2}\right ) t}\T^{\sigma,\psi}(t)\f\|_{r^*}^2$ is  decreasing for every $\f\in L^2(\Omega;\C^m) \cap L^{\frac{2r}{r-2}}(\Omega; \C^m)$, since $(e^{-\tau_{r^*,\sigma} t}\T^{\sigma, \psi}(t))_{t\ge 0}$ is a contraction semigroup on $L^{r^*}(\Omega;\C^m)$
and  $\tau_{r^*,\sigma}\le \hat\tau_{r^*,\sigma}$.
Thus, by applying  \eqref{Sob-1} with $\uu=e^{-\left (\hat\tau_{r^*,\sigma}+\frac{\nu_0}{2}\right ) s}\T^{\sigma,\psi}(s)\f$ and taking into account that 
\begin{align*}
\frac{d}{ds}\|\T^{\sigma,\psi}(s)\f\|_2^2 =- 2\re\af_{\sigma,\psi}(\T^{\sigma, \psi}(s)\f,\T^{\sigma, \psi}(s)\f),\qquad\;\, s\in ]0,\infty[\,,
\end{align*}
we get 
\begin{align*}
&(t-\varepsilon)\|e^{-\left (\hat\tau_{r^*,\sigma}+\frac{\nu_0}{2}\right )t}\T^{\sigma,\psi}(t)\f\|_{r^*}^{\frac{2r}{d}}\\
\le &\int_{\varepsilon}^t \|e^{-\left (\hat\tau_{r^*,\sigma}+\frac{\nu_0}{2}\right )s}\T^{\sigma,\psi}(s)\f\|_{r^*}^{\frac{2r}{d}} ds\\
\leq &  \frac{2\hat c_{r,d}^{\frac{2r}{d}}}{\nu_0}\bigg[\int_{\varepsilon}^te^{-\frac{2r}{d}\left (\hat\tau_{r^*,\sigma}+\frac{\nu_0}{2}\right ) s}\re\af_{\sigma, \psi}(\T^{\sigma, \psi}(s)\f,\T^{\sigma, \psi}(s)\f)\|\T^{\sigma,\psi}(s)\f\|_2^{\frac{2r}{d}-2} ds\\
&\qquad\;\,+\bigg (\hat\tau_{2,\sigma}+\frac{\nu_0}{2}\bigg )\int_{\varepsilon}^t e^{-\frac{2r}{d}\left (\hat\tau_{r^*,\sigma}+\frac{\nu_0}{2}\right ) s} \|\T^{\sigma, \psi}(s)\f\|_2^{\frac{2r}{d}}ds\bigg ]\\
=& \frac{2\hat c_{r,d}^{\frac{2r}{d}}}{\nu_0}
\bigg [-\frac{1}{2}\int_{\varepsilon}^te^{-\frac{2r}{d}\left (\hat\tau_{r^*,\sigma}+\frac{\nu_0}{2}\right ) s}\frac{d}{ds}(\|\T^{\sigma, \psi}(s)\f\|^2)\|\T^{\sigma, \psi}(s)\f\|_2^{\frac{2r}{d}-2} ds\\
&\qquad\;\;+\bigg (\hat\tau_{2,\sigma}+\frac{\nu_0}{2}\bigg )\int_{\varepsilon}^t e^{-\frac{2r}{d}\left (\hat\tau_{r^*,\sigma}+\frac{\nu_0}{2}\right ) s} \|\T^{\sigma, \psi}(s)\f\|_2^{\frac{2r}{d}}ds\bigg ]\\
=& \frac{2\hat c_{r,d}^{\frac{2r}{d}}}{\nu_0}
\bigg [-\frac{d}{2r}\int_{\varepsilon}^te^{-\frac{2r}{d}\left (\hat\tau_{r^*,\sigma}+\frac{\nu_0}{2}\right ) s}\frac{d}{ds}(\|\T^{\sigma, \psi}(s)\f\|_2^{\frac{2r}{d}}) ds\\
&\qquad\;\;+\bigg (\hat\tau_{2,\sigma}+\frac{\nu_0}{2}\bigg )\int_{\varepsilon}^t e^{-\frac{2r}{d}\left   (\hat\tau_{r^*,\sigma}+\frac{\nu_0}{2}\right ) s} \|\T^{\sigma, \psi}(s)\f\|_2^{\frac{2r}{d}}ds\bigg ]\\
= &\frac{2\hat c_{r,d}^2}{\nu_0}\bigg [\frac{d}{2r}e^{-\frac{2r}{d}\left (\hat\tau_{r^*,\sigma}+\frac{\nu_0}{2}\right )\varepsilon} \|\T^{\sigma,\psi}(\varepsilon)\f\|_2^{\frac{2r}{d}}+(\hat\tau_{2,\sigma}-\hat\tau_{r^*,\sigma})\int_{\varepsilon}^t
e^{-\frac{2r}{d}\left (\hat\tau_{r^*,\sigma}+\frac{\nu_0}{2}\right ) s}
\|\T^{\sigma, \psi}(s)\f\|_2^{\frac{2r}{d}} ds\\
&\qquad\quad-\frac{d}{2r}e^{-\frac{2r}{d}\left (\hat\tau_{r^*,\sigma}+\frac{\nu_0}{2}\right ) t}\|\T^{\sigma, \psi}(t)\f\|_2^{\frac{2r}{d}}\bigg ]\\
\le &
\frac{\hat c_{r,d}^2d}{\nu_0r}\|\f\|_2^{\frac{2r}{d}}
\end{align*}
for every $t\in ]\varepsilon,\infty[$\,, since $\hat\tau_{2,\sigma}<\hat\tau_{r^*,\sigma}$ and $e^{-\frac{2r}{d}\left (\hat\tau_{r^*,\sigma}+\frac{\nu_0}{2}\right )\varepsilon} \|\T^{\sigma,\psi}(\varepsilon)\f\|_2^{\frac{2r}{d}}\le\|\f\|_2^{\frac{2r}{d}}$, i.e.,
\begin{align*}
\|\T^{\sigma,\psi}(t)\f\|_{\frac{2r}{r-2}}\le
\bigg (\frac{\hat c^2_{r,d}d}{\nu_0r}\bigg )^{\frac{d}{2r}}
(t-\varepsilon)^{-\frac{d}{2r}}
e^{\left (\hat\tau_{r^*,\sigma}+\frac{\nu_0}{2}\right ) t}\|\f\|_2,\qquad\;\,t\in ]\varepsilon,\infty[\,.
\end{align*}

Letting $\varepsilon$ tend to zero and, then, using a density argument, estimate \eqref{pitbike-2} follows at once with $c_{r,d}=\hat c_{r,d}^{\frac{d}{r}}d^{\frac{d}{2r}}r^{-\frac{d}{2r}}$.

{\em Step 3}.
Here, we prove that the semigroup $(\T^{\sigma, \psi}(t))_{t \ge 0}$ is ultrabounded, i.e., each operator $\T^{\sigma,\psi}(t)$ is bounded from $L^2(\Om;\C^m)$ into $L^\infty(\Om;\C^m)$ and it satisfies the estimate
\begin{equation}
\label{aim_infty}
\|\T^{\sigma, \psi}(t)\|_{\mathcal{L}(L^2(\Omega;\C^m);L^{\infty}(\Omega;\C^m))}\le C_{d,\beta} \nu_0^{-\frac{d}{4}}(t^{-1}+H(|\sigma|^{2\beta+2}+1))^{\frac{d}{4}}e^{\tau_{2,\sigma}t},\qquad\;\,t\in ]0,\infty[,
\end{equation}
for some positive constants $C_{d,\beta}$ and $H=H(d,\beta,c,\kappa,\nu_0)$.

Interpolating estimates \eqref{bryanadams-1} and \eqref{pitbike-2}, thanks to Riesz-Thorin theorem, we deduce that
\begin{align}\label{lisa}
\|\T^{\sigma, \psi}(t)\|_{\mathcal{L}(L^p(\Omega;\C^m);L^{\frac{pr}{r-1}}(\Omega;\C^m))} &\leq \|\T^{\sigma, \psi}(t)\|
_{\mathcal{L}(L^2(\Omega;\C^m);L^{\frac{2r}{r-2}}(\Omega;\C^m))}^{\frac{1}{p}}\! \|\T^{\sigma, \psi}(t)\|_{\mathcal{L}(L^{2(p-1)}(\Omega;\C^m))}^{1-\frac{1}{p}}\notag\\
&\leq  \Big (c_{r,d}\nu_0^{-\frac{d}{2r}}t^{-\frac{d}{2r}}e^{\left (\hat\tau_{r^*,\sigma}+\frac{\nu_0}{2}\right ) t}\Big )^{\frac{1}{p}}
\left(e^{\tau_{2(p-1),\sigma}t}\right)^{1-\frac{1}{p}}\notag\\
&=\big (c_{r,d}^2\nu_0^{-\frac{d}{r}}\big )^{\frac{1}{2p}}t^{-\frac{d}{2pr}}
e^{\rho_{p,\sigma,\nu_0} t}
\end{align}
for every $t\in ]0,\infty[$ and $p\in [2,\infty[$\,, where $\rho_{p,\sigma,\nu_0}=\frac{1}{p}\left (\hat\tau_{r^*,\sigma}+\frac{\nu_0}{2}+(p-1)\tau_{2(p-1),\sigma}\right )$.

Now, similarly to  the proof of 
\cite[Theorem 6.8]{Ou}, we set
$R=(r-1)^{-1}r$, $t_j=
 \frac{(R^{2\beta+1}+1)R-1}{(R^{2\beta+1}+1)R}((R^{2\beta+1}+1) R)^{-j}$ and  $p_j=2R^j$ for every $j\in\N\cup\{0\}$.
Since 
\begin{eqnarray*}
\begin{array}{lll}
\displaystyle\sum_{j=0}^{\infty}t_j=1,
\qquad\;\,&\displaystyle\sum_{j=0}^{\infty}\frac{1}{p_j}=\frac{r}{2},\\[4mm]
\displaystyle\sum_{j=0}^{\infty} \frac{t_j}{p_j} =
\frac{R^2(R^{2\beta+1}+1)-R}{2(R^2(R^{2\beta+1}+1)-1)}
=:A_{r,\beta}<1,\qquad\;\, &
\displaystyle\prod_{j=0}^{\infty} t_j^{-\frac{1}{2p_j}}=:B_{r,\beta}\in ]0,\infty[
\end{array}
\end{eqnarray*}
and
\begin{align*}
\sum_{j=0}^{\infty} \frac{(p_j-1)^{2\beta+3}}{p_j}t_j
\le \sum_{j=0}^{\infty}p_j^{2\beta+2}t_j
% =&2^{2\beta+2}\frac{(R^{2\beta+1}+1)R-1}{(R^{2\beta+1}+1)R}\sum_{j=0}^{\infty}\left (\frac{R^{2\beta+1}}{R^{2\beta+1}+1}\right )^j\\
=&\frac{2^{2\beta+2}}{R}[(R^{2\beta+1}+1)R-1]
=:L_{r,\beta}<\infty,
\end{align*}
thanks to \eqref{lisa}, we can estimate
\begin{align}\label{risate}
&\bigg\|\T^{\sigma,\psi}\bigg (t\sum_{j=0}^{n-1}t_j\bigg )\bigg\|_{\mathcal{L}(L^2(\Omega;\C^m);L^{p_n}(\Omega;\C^m))}\notag\\
\le &\prod_{j=0}^{n-1}\|\T^{\sigma,\psi}(tt_j)\|_{\mathcal{L}(L^{p_j}(\Omega;\C^m);L^{p_{j+1}}(\Omega;\C^m))}\notag\\
\le &\prod_{j=0}^{n-1}(c_{r,d}^2\nu_0^{-\frac{d}{r}}\big )^{\frac{1}{2p_j}}(t_jt)^{-\frac{d}{2p_jr}}
e^{\rho_{p_j,\sigma,\nu_0}tt_j}\notag\\
\le &(c_{r,d}^2\nu_0^{-\frac{d}{r}}\big )^{\sum_{j=0}^{n-1}\frac{1}{2p_j}}
t^{-\frac{d}{r}\sum_{j=0}^{n-1}\frac{1}{2p_j}}
\bigg (\prod_{j=0}^{n-1} t_j^{-\frac{1}{2p_j}}\bigg )^{\frac{d}{r}}
e^{t\sum_{i=0}^{n-1}\rho_{p_i,\sigma,\nu_0}t_i}.
\end{align}

Recalling that $(a+b)^q\le 2^{q-1}(a^q+b^q)$ for every $a,b\in [0,\infty[$ and $q\in ]1,\infty[$\,,  we can easily estimate
\begin{align}\label{est_tau}
\hat\tau_{p, \sigma}&\le
c\bigg(1+\frac{p^2}{2}\bigg)^{\beta+1}(|\sigma|+\kappa)^{2\beta+2}\notag\\
&\le c p^{2\beta+2}2^{2\beta+1}(|\sigma|^{2\beta+2}+\kappa^{2\beta+2})\notag\\
&\le\hat H p^{2\beta+2}(|\sigma|^{2\beta+2}+1),
\end{align}
for every $p \in [2,\infty[$\,, where $\hat H= 2^{2\beta+1}c\max\{1,\kappa^{2\beta+2}\}$; whence, since $\tau_{2(p-1),\sigma}\le \hat{\tau}_{2(p-1),\sigma}$ and $2(p-1)\ge 2$, we can use \eqref{est_tau} twice to get 
\begin{align*}
\sum_{j=0}^{n-1}\rho_{p_j,\sigma,\nu_0}t_j\le &
\sum_{j=0}^{\infty}\rho_{p_j,\sigma,\nu_0}t_j\notag\\
=& \bigg (\hat\tau_{r^*,\sigma}+\frac{\nu_0}{2}\bigg )\sum_{j=0}^\infty \frac{t_j}{p_j}+\sum_{j=0}^\infty\frac{t_j(p_j-1)}
    {p_j}\tau_{2(p_j-1),\sigma}\notag\\
\le &\hat  H (|\sigma|^{2\beta+2}+1)\bigg ((r^*)^{2\beta+2} \sum_{j=0}^\infty\frac{t_j}{p_j} 
+2^{2\beta+2}\sum_{j=0}^\infty\frac{(p_j-1)^{2\beta+3}}{p_j}t_j\bigg )+\frac{\nu_0}{2}\sum_{j=0}^{\infty}\frac{t_j}
{p_j}\notag\\
\le &
(r^*)^{2\beta+2}  \hat H(A_{r,\beta} 
+L_{r,\beta})(|\sigma|^{2\beta+2}+1)+\frac{\nu_0}{2}A_{r,\beta}\notag\\
\le & \max\bigg\{(r^*)^{2\beta+2}  \hat H,\frac{\nu_0}{2}\bigg\}(2A_{r,\beta} 
+L_{r,\beta})(|\sigma|^{2\beta+2}+1)\notag\\
=&\!:H(|\sigma|^{2\beta+2}+1).
\end{align*}

Now, we fix a bounded open set $\Omega_0\subset\Omega$ and $k,n\in\N$, with $k<n$. Using H\"older inequality and \eqref{risate}, we can estimate
\begin{align}
&\bigg\|\T^{\sigma,\psi}\bigg ( t\sum_{j=0}^{n-1}t_j\bigg )\f\bigg\|_{L^{p_k}(\Omega_0;\C^m)}\notag\\
\le &\bigg\|\T^{\sigma,\psi}\bigg ( t\sum_{j=0}^{n-1}t_j\bigg )\f\bigg\|_{L^{p_n}(\Omega_0;\C^m)}
|\Omega_0|^{\frac{1}{p_k}-\frac{1}{p_n}}\notag\\
\le &\bigg\|\T^{\sigma,\psi}\bigg ( t\sum_{j=0}^{n-1}t_j\bigg )\bigg\|_{\mathcal{L}(L^2(\Omega;\C^m);L^{p_n}(\Omega;C^m))}
|\Omega_0|^{\frac{1}{p_k}-\frac{1}{p_n}}\|\f\|_{2}\notag\\
\le & 
(c_{r,d}^2\nu_0^{-\frac{d}{r}})^{\sum_{j=0}^{n-1}\frac{1}{2p_j}}
t^{-\frac{d}{r}\sum_{j=0}^{n-1}\frac{1}{2p_j}}
\bigg (\prod_{j=0}^{n-1} t_j^{-\frac{1}{2p_j}}\bigg )^{\frac{d}{r}}
|\Omega_0|^{\frac{1}{p_k}-\frac{1}{p_n}}e^{H(|\sigma|^{2\beta+2}+1)t}\|\f\|_{2}.
\label{stima-sem}
\end{align}

Note that the function  $\T^{\sigma,\psi}(\cdot)\f$ is continuous in $]0,\infty[$ with values in $L^{p_k}(\Omega;\C^m)$.
Indeed, for every $\varepsilon\in ]0,\infty[$ and $t\in [\varepsilon,\infty[$\,, we can split
$\T^{\sigma,\psi}(t)\f=\T^{\sigma,\psi}(t-\varepsilon)\T^{\sigma,\psi}(\varepsilon)\f$ and the function 
$\T^{\sigma,\psi}(\varepsilon)\f$ belongs to $L^{p_k}(\R^d;\C^m)$ as it can be easily seen, using the previous bootstrap argument, writing
$\T^{\sigma,\psi}(\varepsilon)\f=
\T^{\sigma,\psi}(\varepsilon k^{-1})\cdots\T^{\sigma,\psi}(\varepsilon k^{-1})\f$. The strong continuity of the semigroup $(\T^{\sigma,\psi}(t))_{t\ge 0}$ in $L^{p_k}(\Omega;\C^m)$ implies that the function
$\T^{\sigma,\psi}(\cdot)\f$ is continuous in $[\varepsilon,\infty[$. The arbitrariness of $\varepsilon$ yields the claim.

Letting $n$ tend to $\infty$ in the first and last side of \eqref{stima-sem} and taking all the above arguments into account, we conclude that
\begin{align*}
\|\T^{\sigma,\psi}(t)\f\|_{L^{p_k}(\Omega_0;\C^m)}
\le & c_{r,d}^{\frac{r}{2}}\nu_0^{-\frac{d}{4}}B_{r,\beta}^{\frac{d}{r}}|\Omega_0|^{\frac{1}{p_k}}
t^{-\frac{d}{4}}
e^{H(|\sigma|^{2\beta+2}+1)t}\|\f\|_{2}.
\end{align*}

Next, letting $k$ tend to $\infty$, we infer that the
function $\T^{\sigma,\psi}(t)\f$ is bounded in $\Omega_0$ and
\begin{align*}
\|\T^{\sigma,\psi}(t)\f\|_{L^{\infty}(\Omega_0;\C^m)}
=\lim_{k\to\infty}\|\T^{\sigma,\psi}(t)\f\|_{L^{p_k}(\Omega_0;\C^m)}
\le  c_{r,d}^{\frac{r}{2}}\nu_0^{-\frac{d}{4}}
B_{r,\beta}^{\frac{d}{r}}
t^{-\frac{d}{4}}
e^{H(|\sigma|^{2\beta+2}+1)t}\|\f\|_{2}.
\end{align*}

Finally, observing that $L^{\infty}$-norm of the function $\T^{\sigma,\psi}(t)\f$ is bounded by a constant, independent of $\Omega_0$, we conclude that
\begin{align}
\|\T^{\sigma,\psi}(t)\f\|_{L^{\infty}(\Omega;\C^m)}
\le  c_{r,d}^{\frac{r}{2}}\nu_0^{-\frac{d}{4}}
B_{r,\beta}^{\frac{d}{r}}t^{-\frac{d}{4}}
e^{H(|\sigma|^{2\beta+2}+1)t}\|\f\|_{2}.
\label{corso-2}
\end{align}

We can improve the behaviour of the function $t\mapsto\|\T^{\sigma, \psi}(t)\|_{\mathcal{L}(L^2(\Omega;\C^m);L^{\infty}(\Omega;\C^m))}$ at infinity  taking advantage of \cite[Lemma 6.5]{Ou}, which allows to control the growth of this function by the exponential growth of the function
$t\mapsto\|\T^{\sigma, \psi}(t)\|_{\mathcal{L}(L^2(\Omega;\C^m))}$ modulo an additional polynomial term which grows like $t^{\frac{d}{4}}$ at infinity. More precisely, we can write
\begin{align*}
\|\T^{\sigma,\psi}(t)\f\|_{\infty}
\le  c_{r,d}^{\frac{r}{2}}\nu_0^{-\frac{d}{4}}
B_{r,\beta}^{\frac{d}{r}}e
\big (t^{-1}+H(|\sigma|^{2\beta+2}+1)\big )^{\frac{d}{4}}e^{\tau_{2,\sigma}t}\|\f\|_2,\qquad\;\,t\in [0,\infty[\,.
\end{align*}
whence \eqref{aim_infty} follows with $C_{d,\beta}=c_{r,d}^{\frac{r}{2}}B_{r,\beta}^{\frac{d}{r}}e$ and taking $r=d$, if $d \ge 3$ and $r=3$ if $d=1,2$.

{\em Step 4.} Here, we complete the proof by showing that the semigroup $(\T^{\sigma, \psi}(t))_{t \ge 0}$ is ultracontractive, i.e., each operator $\T^{\sigma, \psi}(t)$ is bounded from $L^1(\Om;\C^m)$ into $L^\infty(\Om;\C^m)$, and by proving the kernel estimate \eqref{ker-est}.

Since $(\af_{\sigma,\psi})^*= (\af^*)_{ -\sigma, \psi}$ and the coefficients of the adjoint form $\af^*$ satisfy the same assumptions as those of the form $\af$, the semigroup $((\T^{\sigma, \psi}(t))^*)_{t\ge 0}$ satisfies estimate \eqref{aim_infty} as well; namely
\begin{align*}
\|(\T^{\sigma, \psi}(t))^*\|_{\mathcal{L}(L^2(\Omega;\C^m);L^{\infty}(\Omega;\C^m))} \leq & C_{d,\beta}\nu_0^{-\frac{d}{4}}
\big (t^{-1}+ H(|\sigma|^{2\beta+2}+1)\big )^{\frac{d}{4}}e^{\tau_{2,\sigma}t},\qquad\;\,t\in ]0,\infty[\,,
\end{align*}
so that, by duality
\begin{align}
\|\T^{\sigma, \psi}(t)\|_{\mathcal{L}(L^1(\Omega;\C^m);L^2(\Omega;\C^m))} \leq & C_{d,\beta}\nu_0^{-\frac{d}{4}}
(t^{-1}+ H(|\sigma|^{2\beta+2}+1))^{\frac{d}{4}}e^{\tau_{2,\sigma}t},\qquad\;\,t\in ]0,\infty[.
\label{corso-3}
\end{align}

Using \eqref{corso-2} and \eqref{corso-3}, we can estimate
\begin{align}
&\|\T^{\sigma, \psi}(t)\|_{\mathcal{L}(L^1(\Omega;\C^m);L^{\infty}(\Omega;\C^m))}\notag\\ \le &\|\T^{\sigma, \psi}(t/2)\|_{\mathcal{L}(L^2(\Omega;\C^m);L^{\infty}(\Omega;\C^m))}
\|\T^{\sigma, \psi}(t/2)\|_{\mathcal{L}(L^1(\Omega;\C^m);L^2(\Omega;\C^m))}\notag\\
\le & C_{d,\beta}^2\nu_0^{-\frac{d}{2}}
\big (2t^{-1}+H(|\sigma|^{2\beta+2}+1)\big )^{\frac{d}{2}}e^{\tau_{2,\sigma}t}\notag\\
\le & 2^{\frac{d}{2}}C_{d,\beta}^2\nu_0^{-\frac{d}{2}}
\big (t^{-1}+ H(|\sigma|^{2\beta+2}+1)\big )^{\frac{d}{2}}e^{2^{2\beta+2}\hat H(|\sigma|^{2\beta+2}+1)t}
\label{stim}
\end{align}
for every $t\in ]0,\infty[$\,.

Taking $\sigma=0$ in \eqref{stim}, it follows that the operator $\T(t)$ is bounded from $L^1(\Omega; \C^m)$ into $L^{\infty}(\Omega;\C^m)$ for every $t>0$. Hence, it is described by a kernel $\bm{k}(t,\cdot, \cdot) \in L^{\infty}(\Omega; \C^{m\times m})$.
Moreover, since
$\T^{\sigma,\psi}(t)=e^{\sigma\psi}\T(t)(e^{-\sigma\psi}\cdot)$ for every $t\in [0,\infty[$\,, the kernel of $\T^{\sigma, \psi}(t)$  is the function $(x,y)\mapsto e^{\sigma \psi(x)} \bm{k} (t,x,y)e^{-\sigma \psi(y)}$ and  satisfies the estimate
\begin{align*}
&e^{\sigma \psi(x)} |k_{ij}(t,x,y)| e^{-\sigma \psi(y)}
\le 2^{\frac{d}{2}}C_{d,\beta}^2\nu_0^{-\frac{d}{2}}
\big (t^{-1}+ H(|\sigma|^{2\beta+2}+1)\big )^{\frac{d}{2}}e^{2^{2\beta+2}\hat H(|\sigma|^{2\beta+2}+1)t}
\end{align*}
for every $i,j\in\{1, \dots, m\}$, for every $t\in ]0,\infty[$ and for almost every $x,y\in\Omega$.
Therefore,
\begin{align*}
|k_{ij}(t,x,y)| \leq 2^{\frac{d}{2}}C_{d,\beta}^2\nu_0^{-\frac{d}{2}}
\big (t^{-1}+ H(|\sigma|^{2\beta+2}+1)\big )^{\frac{d}{2}}e^{2^{2\beta+2}\hat H(|\sigma|^{2\beta+2}+1)t+\sigma(\psi(y)-\psi(x))}
\end{align*}
for every $t\in ]0,\infty[$ and for almost every $x, y \in \Omega$.
Taking 
\begin{eqnarray*}
\sigma = (2^{2\beta +2}(2\beta+2)\hat Ht)^{-\frac{1}{2\beta+1}}|\psi(x)-\psi(y)|^{-\frac{2\beta}{2\beta+1}}(\psi(x)-\psi(y))
\end{eqnarray*}
 we obtain
\begin{align*}
|k_{ij}(t,x,y)|\leq &
2^{\frac{d}{2}}C_{d,\beta}^2\nu_0^{-\frac{d}{2}}
\Bigg (t^{-1}+ H+ H \left(\frac{|\psi(x)-\psi(y)|}{2^{2\beta +2}\hat H(2\beta+2)t}\right)^{\frac{2\beta+2}{2\beta+1}}\Bigg )^{\frac{d}{2}}e^{2^{2\beta+2}\hat Ht}\\
&\qquad\times \exp\bigg ({-\displaystyle{\frac{2\beta+1}{2\beta+2}} \frac{|\psi(x)-\psi(y)|^{\frac{2\beta+2}{2\beta+1}}}{(2^{2\beta +2}\hat H(2\beta+2)t)^{\frac{1}{2\beta+1}}}}\bigg )\\
&\leq 2^{\frac{d}{2}}C_{d,\beta}^2\nu_0^{-\frac{d}{2}}
\bigg (t^{-1}+ H+ H_1 \big (t^{-1}d_{Q,V,\beta}(x,y)\big )^{\frac{2\beta+2}{2\beta+1}}\bigg )^{\frac{d}{2}}e^{2^{2\beta+2}\hat Ht}\\
&\qquad\times \exp\bigg ({-\displaystyle{\frac{2\beta+1}{2\beta+2}} \frac{|\psi(x)-\psi(y)|^{\frac{2\beta+2}{2\beta+1}}}{(2^{2\beta +2}\hat H(2\beta+2)t)^{\frac{1}{2\beta+1}}}}\bigg )
\end{align*}
for every $t\in ]0,\infty[$ and for almost every $x, y \in \Omega$. Here, $H_1$ is a constant depending on $\beta$, $d$, $\kappa$, $c$ and $\nu_0$.
 By minimizing over $\psi$, estimate \eqref{ker-est} follows.
\end{proof}

\begin{coro} Under the assumptions of Theorem $\ref{Gauss}$, further suppose that 
$Q\in L^\infty(\Omega;\R^{d\times d})$ 
and  $B_h, C_h\in L^\infty(\Omega;\R^{m\times m})$ for every $h=1, \dots, d$.
Then, the  kernel $\bm{k}(t,\cdot, \cdot) \in L^\infty(\Omega\times \Omega; \C^{m\times m})$ of the semigroup $(\T(t))_{t\ge 0}$ satisfies the estimate
\begin{align}
|k_{ij}(t,x,y)|\leq 
C_0\left(1+t^{-1}+t^{-2}|x-y|^2\right )^{\frac{d}{2}}e^{C_1t-C_2\frac{|x-y|^2}{t}}
\label{ker-est-1}
\end{align}
for every $i,j=1, \ldots,m$, for almost every $x, y \in \Omega$ and for some positive constants $C_0$, $C_1$ and $C_2$.

Moreover, if $\Omega=\Rd$, then there exist $\omega>0$ and $\mu_0 \in [0,\frac\pi 2[$ such that  the operator $A_p+\omega$, where $A_p$  is the generator of the extrapolated semigroup to $L^p(\Rd, \Cm)$,  admits a bounded $H^\infty(\Sigma_\mu)$-calculus in $L^p(\Rd, \Cm)$ for any $\mu\in [\mu_0,\pi)$ and every $p\in (1, \infty)$.
\end{coro}

\begin{proof}
Regarding the Gaussian estimate, it is sufficient to observe that, due to the boundedness of $B_h$ and $C_h$, a suitable choice of the constants $\kappa$ and $c$ in Hypothesis \ref{base-2}(iv) allows one to take $\beta=0$. In this case, the boundedness of $Q$ implies that $d_{Q,V}$ is equivalent to the Euclidean metric. Estimate \eqref{ker-est-1} follows.

Now, we consider the second part of the statement and observe that, since $\beta=0$,  under the assumptions of Theorem \ref{Gauss}, there exists $\tilde c>c\gamma_2^{-1}$ (where $c$ is the constant in Hypothesis \ref{base-2}(iv)), such that the operator $A_\V+\tilde c$   is invertible and $m$-accretive in $L^2(\Rd, \Cm)$. Moreover, there exists $\mu_0\in [0,\frac\pi 2[$ such that the  numerical range  of $A_\V+\tilde c$ is contained in the sector
$\Sigma_{\mu_0}=\{\lambda\in\C: |\text{arg}(\lambda)| \leq \mu_0\}$ 
(see e.g.,  \cite[Proposition 7.3.4]{Haase}).
As a consequence, by  a result of Crouzeix and Delyon \cite{CD-1} (see also \cite[Corollary 7.1.17]{Haase}), the operator $A_\V+\tilde c$ admits a bounded  $H^\infty(\Sigma_\mu)$-calculus in $L^2(\Rd, \Cm)$ for every $\mu\in [\mu_0, \pi]$. 

Finally, by possibly choosing a suitable $\omega>\tilde c$, we obtain that the kernel $\tilde{\bm k}(t, \cdot, \cdot)$ of the semigroup generated by $-\A_\V-\omega$ satisfies a Gaussian estimate of the type:
\begin{eqnarray*} 
|\tilde k_{ij}(t,x,y)|\leq 
\tilde C_0t^{-\frac d 2}e^{-\tilde C_2\frac{|x-y|^2}{t}},
\end{eqnarray*}
for every $t>0$ and for almost every $x,y\in\Omega$. Now, the assertion follows by applying \cite[Theorem 3.6]{ALLR}. 
\end{proof}

\section{Examples}
\label{sect-5}

In this section, we provide some classes of systems of elliptic operators to which the main results of the paper apply.

Let $\Omega$ be an open subset of $\Rd$.

\subsection{Assumptions on $Q$}
\label{subsect-5.1}
Let $Q$ be a matrix-valued function such that
\begin{enumerate}[\rm (i)]
\item 
its coefficients $q_{hk}$ ($h,k=1,\ldots,d$) belong to $L^\infty_{{\rm loc}}(\Omega)$;
\item 
for almost every $x\in\Omega$, the matrix $Q(x)$ is symmetric; 
\item 
$\lambda_Q$ is a locally strictly positive function. 
\end{enumerate}

\subsection{Assumptions on $A^{hk}=\bm{(a^{hk}_{ij})_{i,j=1,\ldots,m}}$}
\label{subsect-5.2}
Let us consider two different situations (in both of them the coefficients $a^{hk}_{ij}$ are supposed to belong to $L^\infty_{{\rm loc}}(\Omega)$ for every $h,k=1,\ldots,d$ and every $i,j=1,\ldots,m$):
\begin{enumerate}
\item  
there exists a positive constant $k_0$ such that
\begin{equation}\label{cond_ii}
|a^{hk}_{ij}(x)|\le k_0 \lambda_Q(x), \qquad \;\, \textit{ a.e.}\,\,x\in \Om,
\end{equation}
for every $i,j=1, \ldots, m$ and $h,k=1, \ldots,d$, and 
\begin{equation}
\label{cond_i}
\sum_{h,k=1}^d(A^{hk}(x)\theta^k,\theta^h)\ge 0,\qquad\;\, \textit{ a.e.}\,\, x \in \Om,
\end{equation}
for every $\theta^k, \theta^h\in \R^m$;
\item 
for every $h,k=1,\ldots,d$, 
$A^{hk}=q_{hk}G$, 
where $G$ is a $m\times m$ matrix-valued function having measurable, nonnegative and bounded entries $g_{ij}:\Om \to \R$,  and
$(G(x)\xi, \xi)\ge 0$ for almost everywhere $x \in \Om$ and $\xi \in \R^m.$
\end{enumerate}

In the first case, Hypothesis \ref{base}(iii) is satisfied with $\kappa_A=mdk_0 $. Indeed, from \eqref{cond_i} we get immediately that
\begin{align*} 
{\rm Re}\sum_{h,k=1}^d(A^{hk}(x)\theta^k,\theta^h)\ge 0, \qquad\;\, \theta^h, \theta^k\in \mathbb{C}^m,
\end{align*}
for almost every $x\in \Om$. 
Further, using \eqref{cond_ii} we can estimate
\begin{align*}
{\rm Re}\sum_{h,k=1}^d(A^{hk}\theta^k,\theta^h) 
&\le  \sum_{h, k=1}^d\sum_{i,j=1}^m  |a^{hk}_{ij}||\theta^k_j| |\theta^h_i|\notag\\
& \le k_0\lambda_Q\bigg(\sum_{k=1}^d\sum_{i=1}^m|\theta^k_i|\bigg)^2\notag\\
& \le md k_0  \lambda_Q\sum_{k=1}^d\sum_{i=1}^m|\theta^k_i|^2\notag\\
& \le mdk_0\mathcal{Q}(\theta^1,\ldots,\theta^d)
\end{align*}
for every $\theta^h=(\theta^h_1,\ldots, \theta^h_m), \theta^k=(\theta^k_1,\ldots, \theta^k_m),  \in \Cm$ and almost everywhere in $\Om$, whence \eqref{realLegendrefunc} follows with $\kappa_A=mdk_0$.
Further, using again \eqref{cond_ii}, we can estimate
\begin{align*}
\bigg |{\rm  Im}\sum_{h,k=1}^d(A^{hk}\theta^k,\theta^h)\bigg |
\le &   
 \sum_{h,k=1}^d\sum_{i,j=1}^m |a^{hk}_{ij}|
 \big(|{\rm Im}\hskip 1pt\theta^k_j||{\rm Re}\hskip 1pt\theta_i^h|+ |{\rm Re}\hskip 1pt\theta^k_j|{\rm Im}\hskip 1pt\theta^h_i|\big)\notag\\
\le &\frac{1}{2}\sum_{h,k=1}^d\sum_{i,j=1}^m |a^{hk}_{ij}|\big(|\theta^k_j|^2+|\theta_i^h|^2\big)\notag\\
\le & mdk_0 \lambda_Q\sum_{k=1}^d\sum_{i=1}^m|\theta^k_i|^2\notag\\
\le &\kappa_A\mathcal{Q}(\theta^1,\ldots,\theta^d),
\end{align*}
whence also condition \eqref{comLeg} is satisfied.

Now, we consider the second situation.
In order to check conditions \eqref{realLegendrefunc} and \eqref{comLeg}, we observe that 
\begin{eqnarray*}
\sum_{h,k=1}^d(A^{hk}\theta^k,\theta^h) = \sum_{h,k=1}^d\sum_{i,j=1}^m q_{hk}g_{ij}\theta^k_i\overline\theta^h_j= \sum_{i,j=1}^m g_{ij} (Q \theta_i,\theta_j)
\end{eqnarray*}
for every $\theta^1,\ldots,\theta^d\in\Cm$,
where by $\theta_i$ ($i=1,\ldots,m$) we denote the vector in $\mathbb C^d$ having coordinates $(\theta_i^1, \ldots, \theta_i^d)$. Thus, we obtain that
\begin{align}
\re \sum_{h,k=1}^d(A^{hk}\theta^k,\theta^h)
=\sum_{i,j=1}^mg_{ij}\big((Q^{\frac{1}{2}}\re \theta_i, Q^{\frac{1}{2}}\re \theta_j)+(Q^{\frac{1}{2}}\im \theta_i, Q^{\frac{1}{2}}\im \theta_j)\big ),
\label{pril}
\end{align}
which is nonnegative by the assumption on the matrix-valued function $G$. Moreover, since 
$(G(x)\xi,\xi)\le \Lambda_G|\xi|^2$ for every $\xi \in \R^m$, for some positive constant $\Lambda_G$ and for almost every in $x\in\Om$, and
\begin{align*}
\sum_{i,j=1}^m g_{ij}(Q^{\frac{1}{2}}\xi_i, Q^{\frac{1}{2}}\xi_j)
\le \sum_{i,j=1}^d g_{ij}(x)
|Q^{\frac{1}{2}}\xi_i||Q^{\frac{1}{2}}\xi_j|
\le \Lambda_G  \sum_{i=1}^d |Q^{\frac{1}{2}}\xi_i|^2
=\Lambda_G(Q\xi,\xi),
\end{align*}
for every $\xi \in \R^m$, from \eqref{pril} we conclude that
\begin{align*}
\re\sum_{h,k=1}^d(A^{hk}\theta^k,\theta^h)
\le \Lambda_G \mathcal{Q}(\theta^1,\ldots,\theta^d).
\end{align*}

Analogously, taking into account that $g_{ij}$ are nonnegative functions, we can estimate
\begin{align*}
\bigg |\im \sum_{h,k=1}^d(A^{hk}\theta^k,\theta^h)\bigg |\le\sum_{i,j=1}^m g_{ij}|\im (Q\theta_i, \theta_j)|
\le \sum_{i,j=1}^m g_{ij}|Q^{\frac{1}{2}}\theta_i||Q^{\frac{1}{2}}\theta_j|
\le 
%\Lambda_G \sum_{i=1}^m|Q^{\frac{1}{2}}\theta_i|^2= 
\Lambda_G \mathcal{Q}(\theta^1,\ldots,\theta^d),
\end{align*}
whence conditions \eqref{realLegendrefunc} and \eqref{comLeg} are both satisfied with $\kappa_A=\Lambda_G$.

\subsection{Assumptions on $B^h$ and $C^h$}
Let us assume that $B^h$ and $C^h$, $h=1,\ldots,d$, belong to $L^\infty_{{\rm loc}}(\Omega;\R^{m\times m})$ and that there exists a constant $\kappa\in\, ]0,d^{-\frac{1}{2}}[$ such that 
\begin{eqnarray*} 
\max\{\|B^h(x)\|, \|C^h(x)\|\}\le \kappa(\lambda_Q(x)\lambda_V(x))^{\frac{1}{2}}, \qquad {\rm a.e.}\,\,x \in \Om.
\end{eqnarray*}
In this case, arguing as in the forthcoming Subsection \ref{subsect-5.6}, we conclude that  estimates \eqref{est_1} and \eqref{est_2} are satisfied with $\kappa_B=\kappa_C=\kappa\sqrt{d}$, $\gamma=1$ and any positive constant $C_\gamma$.
% Indeed, 
% \begin{align*}
% \max\bigg\{\bigg |\sum_{h=1}^d (B^h\eta,\theta^h) \bigg|, \bigg |\sum_{h=1}^d (C^h\eta,\theta^h) \bigg| \bigg\}\leq &\kappa\sum_{h=1}^d(\lambda_Q\lambda_V)^{\frac{1}{2}} |\theta^h| |\eta| \\
% \le & \kappa\sqrt{d}\bigg(\sum_{h=1}^d \lambda_Q|\theta^h|^2\bigg)^\frac 1 2(\lambda_V|\eta|^2)^{\frac{1}{2}} \\
% \le &\kappa\sqrt{d}\bigg(\sum_{i=1}^m \sum_{h=1}^d  \lambda_Q |\theta_i^h|^2\bigg)^\frac 1 2  (V_S\eta, \eta)^{\frac{1}{2}}\\
% \le & \kappa\sqrt{d}\big(\mathcal{Q}(\theta^1,\ldots,\theta^d)\big )^\frac 1 2  \left((V_S\eta, \eta)+ C|\eta|^2\right)^{\frac{1}{2}}.
% \end{align*}

\subsection{Assumptions on $V$} 
\label{subsect-5.3}
As far as the potential term it concerned, let us assume that the matrix-valued function $V$ belongs to $L^\infty_{\rm loc}(\Omega; \R^{m \times m})$ and satisfies one of the following conditions:
\begin{enumerate}[\rm(1)]
\item 
$V(x)$ is a positive definite and symmetric real-valued matrix for almost every $x \in \Om$. with $\operatorname{essinf}_{x \in \Om}\lambda_V(x)=:v_0>0$;
\item 
$V(x)$ is a diagonal perturbation of an antisymmetric matrix-valued function for almost every $x \in \Om$, i.e., 
\begin{eqnarray*}
V(x)={\rm diag}(v_{11}(x), \ldots, v_{mm}(x))+ V_0(x),\qquad\;\,\textit{a.e. } x \in \Om,
\end{eqnarray*}
where $V_0(x)$ is an $m\times m$ antisymmetric matrix and there exist positive constants $k_1,k_2$ such that
\begin{equation*}
\operatorname{essinf}_{x \in \Om}v_{ii}(x)>k_1,\qquad\;\, i=1, \ldots,m,
\end{equation*}
and
\begin{equation*}
\sum_{j\in\{1,\ldots,m\}\setminus\{i\}}|v_{ij}(x)|\le  k_2v_{ii}(x),\qquad\;\,  i=1, \ldots, m,\,\,  \textit{a.e. }x \in \Om.
\end{equation*}
\end{enumerate}

Clearly, in the first case $V(x)=V_S(x)$ and  ${\rm Im}(V(x)\xi,\xi)=0$ for almost every $x \in \Om$ and every $\xi \in \C^m$; hence Hypothesis \ref{base}(ii) is satisfied.\\
In the second case
\begin{equation*}
{\rm Re}(V(x)\xi,\xi)= \sum_{i=1}^mv_{ii}(x)|\xi_i|^2>k_1|\xi|^2
\end{equation*}
and
\begin{equation*}
{\rm Im} (V(x)\xi,\xi)=\sum_{i=1}^m\sum_{j\in\{1,\ldots,m\}\setminus\{i\}}v_{ij}(x)\xi_j\overline{\xi}_i
\end{equation*}
for almost every $x\in \Om$ and every $\xi=(\xi_1,\ldots,\xi_m)\in\mathbb C^m$.
Thus, it follows that
\begin{align*}
|{\rm Im} (V(x)\xi,\xi)|
& \le \frac{1}{2}\sum_{i=1}^m\sum_{j\in\{1,\ldots,m\}\setminus\{i\}}|v_{ij}(x)|(|\xi_i|^2+|\xi_j|^2)\\
&= \sum_{i=1}^m|\xi_i|^2\sum_{j\in\{1,\ldots,m\}\setminus\{i\}}|v_{ij}(x)|\\
&\le k_2\re (V(x)\xi,\xi)
\end{align*}
for every $\xi\in\mathbb C^m$ and almost every  $x\in\Om$,
whence Hypothesis \ref{base}(ii) is satisfied with $c_0=k_2$.

\subsection{Assumptions on $W$}
Concerning the matrix-valued function $W\in L^\infty_{\rm loc}(\Om;\R^{m\times m})$, let us assume that 
\begin{eqnarray*}
\|W(x)\|\le \kappa_W\lambda_V(x),\qquad\;\,\textit{a.e. } x\in\Omega,    
\end{eqnarray*} 
for some constant $\kappa_W\in ]0,1-\kappa^2d[$\,. In such a case, estimate \eqref{est_3} is satisfied with $\gamma=1$ and any positive $C_\gamma$. 

\medskip

Under all the previous conditions on $Q$, $B^h$, $C^h$ ($h=1,\ldots,d$), $V$ and $W$, Hypotheses \ref{base} are satisfied. 

\medskip

To fulfill Hypotheses \ref{5bis}, the assumptions on the matrix-valued functions $B_h$, $C_h$ ($h=1,\ldots,d$) and $W$, should be refined as shown here below.

\subsection{Refined assumptions on $B_h$ and $C_h$}
\label{subsect-5.6}
Let us assume that the matrix-valued functions $B^h$ and $C^h$ satisfy the conditions 
\begin{equation} \label{allgamma}
\max\{\|B^h(x)\|, \|C^h(x)\|\}\leq \kappa (\lambda_Q(x))^{\frac{1}{2}} (\lambda_V(x))^a,\qquad\;\, h=1, \ldots,d,
\end{equation}
for almost every $x \in \Om$ and some positive constant $\kappa$.

Observing that
\begin{eqnarray*}
x^{2a} \leq \gamma x + (1-2a)\left(\frac{\gamma}{2a}\right)^{\frac{2a}{2a-1}},\qquad\;\, 
x\in [0,\infty[\,,\;\,\gamma\in ]0,\infty[\,,
\end{eqnarray*}
if $a\in \left ]0,\frac{1}{2}\right [$ and setting $\varphi(\gamma)=(1-2a)\left(\frac{\gamma}{2a}\right)^\frac{2a}{2a-1}$, we get 

\begin{align*}
\max\bigg\{\bigg |\sum_{h=1}^d (B^h\eta,\theta^h) \bigg|, \bigg |\sum_{h=1}^d (C^h\eta,\theta^h) \bigg| \bigg\}\leq &\kappa\lambda_Q^{\frac{1}{2}}\lambda_V^a|\eta|\sum_{h=1}^d |\theta^h| \\
\le & \kappa\sqrt{d}\bigg(\sum_{h=1}^d \lambda_Q|\theta^h|^2\bigg)^\frac 1 2(\lambda_V^{2a}|\eta|^2)^{\frac{1}{2}} \\
\le &\kappa\sqrt{d}\bigg(\sum_{i=1}^m \sum_{h=1}^d  \lambda_Q |\theta_i^h|^2\bigg)^\frac 1 2    \left(\gamma (V_S\eta, \eta)+ \varphi(\gamma)|\eta|^2\right)^{\frac{1}{2}}\\
\le & \kappa\sqrt{d}\big(\mathcal{Q}(\theta^1,\ldots,\theta^d)\big )^\frac 1 2  \big (\gamma (V_S\eta, \eta)+ \varphi(\gamma)|\eta|^2\big )^{\frac{1}{2}}
\end{align*}
almost everywhere in $\Om$. Thus, estimates \eqref{est_11} and \eqref{est_22} in Hypotheses \ref{5bis} are satisfied.

Finally, we observe that, if \eqref{allgamma} holds with $a=0$, then, taking  condition (1) in Subsection \ref{subsect-5.3} into account, we estimate
\begin{eqnarray*}
\max\{\|B^h(x)\|, \|C^h(x)\|\}\leq \kappa v_0^{-a} (\lambda_Q(x))^{\frac{1}{2}} (\lambda_V(x))^a,\qquad\;\, h=1, \ldots,d,
\end{eqnarray*}
for almost every $x\in\Omega$ and every $a\in \left ]0,\frac{1}{2}\right [$\,.

\subsection{Refined assumptions on $W$} 
\label{subsect-5.7}
Let us assume that 
\begin{equation*}
\|W(x)\|\le \kappa_0 (\lambda_V(x))^{b},\qquad\;\,\textit{a.e. } x\in\Omega,
\end{equation*}
for some constants $k_0\in ]0,\infty[$ and some $b \in [0,1[$\,. Arguing as in Subsection \ref{subsect-5.6}, it can be easily shown that
\begin{eqnarray*}
|(W(x)\xi,\eta)|\le \kappa_0\big (\gamma (V_S(x)\xi,\xi)+\widehat\varphi(\gamma)|\xi|^2\big )^{\frac{1}{2}}\big (\gamma (V_S(x)\eta,\eta)+\widehat\varphi(\gamma)|\eta|^2\big )^{\frac{1}{2}}      
\end{eqnarray*}
for every $\xi,\eta\in \R^m$ and every $\gamma\in ]0,\infty[\,$, where $\widehat\varphi(\gamma)=(1-b)\left (\frac{\gamma}{b}\right )^{\frac{b}{b-1}}$.
\medskip

Under the previous assumptions on $Q$, $A^{hk}$, $B^h$, $C^h$ ($h,k=1,\ldots,d)$, $V$ and $W$,  
in Subsections \ref{subsect-5.1}, \ref{subsect-5.2}, 
\ref{subsect-5.3}, \ref{subsect-5.6} and \ref{subsect-5.7}, Hypotheses \ref{5bis} are satisfied with $\kappa_B=\kappa_C=\kappa\sqrt{d}$, $\kappa_W=\kappa_0$ and $\varphi(\gamma)=\max\{(1-2a)(2a)^{\frac{2a}{1-2a}}\gamma^{\frac{2a}{2a-1}}, (1-b)b^{\frac{b}{1-b}}\gamma^{\frac{b}{b-1}}\}$ for every $\gamma\in\, ]0,\infty[$\,. 

In particular, if $b=2a$ and $\mathcal{V}$ satisfies Hypotheses \ref{V} then estimate \eqref{est_norm} holds true for every $p\in ]1,\infty[$\,, if $\kappa_A=0$, and for every $p \in \left ]1+\frac{2\kappa_A}{2\kappa_A+1},2+\frac{1}{2\kappa_A}\right [$ otherwise and
\begin{equation*} 
\gamma_p^{-1}\varphi(\gamma_p)=\frac{1-2a}{(2a)^{\frac{2a}{2a-1}}}\left(\kappa_W+\frac{p^2\kappa^2d}{4[((p-1)^2\wedge 1)-2\kappa_A((p-1)\wedge 1)|p-2|]}\right)^{\frac{1}{1-2a}}.
\end{equation*}

Finally, if, in addition to the previous assumptions, $\Om$ has the extension property, $\operatorname{essinf}_\Om \lambda_Q>0$ and $A^{hk}=0$ for every $h,k=1, \ldots,d$ then Hypotheses \ref{base-2} are satisfied and Theorem \ref{Gauss} can be applied.

\appendix

\section{A Technical result}
In this appendix, we collect a technical result that is used in the paper.

\begin{lemm}
\label{lemma-A1}
For every $\uu\in W^{1,2}(\Omega;\C^m)$, the function
$(|\uu|\wedge 1)\operatorname{sign}\hskip 1pt\uu$ belongs to $W^{1,2}(\Omega;\Cm)$ and
\begin{align}
D_k[(|\uu|\wedge 1)\operatorname{sign}\hskip 1pt\uu]=&
\frac{\re(D_k\uu,\uu)\uu}{|\uu|^2}\chi_{\{|\uu|<1\}}+(|\uu|\wedge 1)\bigg (\frac{D_ku_j}{|\uu|}-\frac{u_j\re(\uu,D_k\uu)}{|\uu|^3}\bigg )\notag\\
=&-\frac{(\operatorname{sign}\hskip 1pt \uu)D_k|\uu|}{|\uu|}\chi_{\{|\uu|>1\}}
+\frac{|\uu|\wedge 1}{|\uu|}D_k\uu
\label{form-A1}
\end{align}
for every $k=1,\ldots,d$.
\end{lemm}

\begin{proof}
We fix $\uu\in W^{1,2}(\Omega;\C^m)$ and split the proof into three steps. 

{\em Step 1}. We claim that the function $|\uu|$ belongs to $W^{1,2}(\Omega)$. This can be easily proved by a standard approximation argument, that we sketch. First, we approximate $\uu$ by a sequence $(\uu_n)\subset C^{\infty}_c(\Omega;\C^m)$ such that
$\uu_n$ converges to $\uu$ in $L^2(\Omega;\C^m)$ and $D_k\uu_n$ converges to $D_k\uu$ in $L^2(\Omega';\C^m)$
for every  $k=1,\dots,d$ and every open set $\Omega' \Subset \Omega$. Moreover, we can assume that these convergences hold also almost everywhere in $\Omega$.

For each $n\in\N$ and $\varepsilon>0$, the function $\uu_{n,\varepsilon}=(|\uu_n|^2+\varepsilon^2)^{\frac{1}{2}}-\varepsilon$
belongs to $C^1(\Omega)$ and has compact support in $\overline{\Omega}$. Hence, it belongs to $W^{1,2}(\Omega)$. Since $|\uu_{n,\varepsilon}-(|\uu|^2+\varepsilon^2)^{\frac{1}{2}}+\varepsilon|\le |\uu_n-\uu|$ for every $n\in\N$ and $\varepsilon>0$, it follows that $\uu_{n,\varepsilon}$ converges to $\uu_{\varepsilon}=(|\uu|^2+\varepsilon^2)^{\frac{1}{2}}-\varepsilon$ in $L^2(\Omega)$. Moreover,
$D_k\uu_{n,\varepsilon}=\re (D_k\uu_n,\uu_n)(\uu_{n,\varepsilon}+\varepsilon)^{-1}$ in $\Omega$ for every $k=1,\ldots,d$ and 
\begin{eqnarray*}
|D_{k}\uu_{n,\varepsilon}-\re (D_k\uu,\uu)(\uu_{\varepsilon}+\varepsilon)^{-1}|
\le  3\varepsilon^{-1}\big (|D_k\uu_n-D_k\uu||\uu_n|+|D_k\uu||\uu_n-\uu|\big )
\end{eqnarray*}
for every $n\in\N$ and $\varepsilon>0$. This is enough to infer that the function $\bm{u}_{\varepsilon}$ admits weak derivative
$D_k\uu_{\varepsilon}=\re(D_k\uu,\uu)(\uu_{\varepsilon}+\varepsilon)^{-1}$, which belongs to $L^2(\Omega)$. Hence, $|\uu_{\varepsilon}|$ belongs to $W^{1,2}(\Omega)$.
As $\varepsilon$ tends to zero, $\uu_{\varepsilon}$ converges to $|\uu|$ in $L^2(\Omega)$ and $D_k\uu_{\varepsilon}$ converges pointwise in $\Omega$ to $\re (D_k\uu,\uu)|\uu|^{-1}$ for every $k\in\{1,\ldots,d\}$.
Moreover, $|D_k\uu_{\varepsilon}|\le |D_k\uu|$ in $\Omega$ for every $\varepsilon>0$. Hence, by dominated convergence, we conclude that $|\uu|$ admits weak derivative $D_k|\uu|$, for every $k$ as above, which belongs to $L^2(\Omega)$. This implies that $|\uu|\in W^{1,2}(\Omega)$.

{\em Step 2.} Next, we show that the function $|\uu|\wedge 1$ belongs to $W^{1,2}(\Omega)$. Since
$|\uu|\wedge 1=|\uu|-(|\uu|-1)^+$, it suffices to show that $(|\uu|-1)^+\in W^{1,2}(\Omega)$. For this purpose, for every $\varepsilon>0$, we consider the function $\varphi_{\varepsilon}\in C^1(\R)$, defined by
$\varphi_{\varepsilon}(t)=(\sqrt{(t-1)^2+\varepsilon^2}-\varepsilon)\chi_{]1,\infty[}(t)$ for every $t\in\R$. The function $\varphi_{\varepsilon}(|\uu|)$ belongs to $W^{1,2}(\Omega)$ and
\begin{eqnarray*}
D_k\varphi_{\varepsilon}(|\uu|)
=(D_k|\uu|)(|\uu|-1)[(|\uu|-1)^2+\varepsilon^2]^{-\frac{1}{2}}\chi_{\{|\uu|>1\}}
\end{eqnarray*}
for every $k\in\{1,\ldots,d\}$ and $\varepsilon\in ]0,\infty[\,$ as it can be proved arguing as above, approximating the function $|\uu|$ by a sequence $(\zeta_n)\subset C^{\infty}_c(\overline{\Omega})$ which converges, as $n$ tends to $\infty$, to $|\uu|$ in $L^2(\Omega)$ and in $W^{1,2}(\Omega')$ for every $\Omega'\Subset\Omega$. We can also assume that $\zeta_n$ and $D_k\zeta_n$ converge pointwise almost everywhere in $\Omega$ to $|\uu|$ and to $D_k|\uu|$, respectively, for every $k\in\{1,\ldots,d\}$. 
Since
\begin{eqnarray*}
|\varphi_{\varepsilon}(|\uu|)-(|\uu|-1)^+|
\le \frac{2\varepsilon(|\uu|-1)}{\sqrt{(|\uu|-1|)^2+\varepsilon^2}+(|\uu|-1+\varepsilon)}\chi_{\{|\uu|>1\}}
\le 2\varepsilon\chi_{\{|\uu|>1\}},
\end{eqnarray*}
$\varphi_{\varepsilon}(|\uu|)$ converges to $(|\uu|-1)^+$ in $L^2(\Omega)$ as $\varepsilon$ tends to $0$. Further, 
$D_k\varphi_{\varepsilon}(|\uu|)$ converges to the function $(D_k|\uu|)\chi_{\{|\uu|>1\}}\in L^2(\Omega)$ as $\varepsilon$ tends to $0$, in a dominated way. Hence, $(|\uu|-1)^+$ admits weak derivative $D_k(|\uu|-1)^+$, which belongs to 
$L^2(\Omega)$. We have so proved that $(|\uu|-1)^+$ and $|\uu|\wedge 1$ belong  to $W^{1,2}(\Omega)$ and the above computations also show that
\begin{eqnarray*}
D_k(|\uu|\wedge 1)=(D_k|\uu|)\chi_{\{|\uu|<1\}}=|\uu|^{-1}\re(D_k\uu,\uu)\chi_{\{|\uu|<1\}},\qquad\;\, k=1,\ldots,d.
\end{eqnarray*}

{\em Step 3.} Now, we complete the proof,  proving that the function $(|\uu|\wedge 1)\operatorname{sign}\hskip 1pt\uu$ belongs to $W^{1,2}(\Omega;\C^m)$. Clearly, it suffices to prove that $(|\uu|\wedge 1)u_j|\uu|^{-1}\in W^{1,2}(\Omega)$ for every $j\in\{1,\ldots,m\}$. We fix such a $j$ and introduce the function $\phi_{\varepsilon}\in C^1(\R)$ defined by $\phi_{\varepsilon}(t)=(t^2+\varepsilon)^{-\frac{1}{2}}$ for every $t\in\R$. The same arguments applied in Step 2, show that the function
$\phi_{\varepsilon}|\uu|)$ belongs to $W^{1,2}(\Omega)\cap L^{\infty}(\Omega)$. Hence, the function $(|\uu|\wedge 1)u_j\phi_{\varepsilon}(|\uu|)$ belongs to $W^{1,2}(\Omega)$ as well, and it converges to $(|\uu|\wedge 1)u_j|\uu|^{-1}$ in $L^2(\Omega)$ as $\varepsilon$ tends to $0$. Moreover, $D_k[(|\uu|\wedge 1)\phi_{\varepsilon}(|\uu|)]$ converges pointwise in $\Omega$ to the function 
\begin{eqnarray*}
\frac{\re(D_k\uu,\uu)u_j}{|\uu|^2}\chi_{\{|\uu|<1\}}+(|\uu|\wedge 1)\bigg (\frac{D_ku_j}{|\uu|}-\frac{u_j\re(\uu,D_k\uu)}{|\uu|^3}\bigg ),
\end{eqnarray*}
which belongs to $L^2(\Omega)$ and
$|D_k[(|\uu|\wedge 1)\varphi_{\varepsilon}(|\uu|)]|\le 3|D_k\uu|$ for every $\varepsilon>0$. By dominated convergence, we conclude that the function $(|\uu|\wedge 1)u_j|\uu|^{-1}$ admits first-order weak derivative along the direction $k$, which is in $L^2(\Omega)$. This is enough to infer that the function $(|\uu|\wedge 1)u_j|\uu|^{-1}\in W^{1,2}(\Omega)$ and formula \eqref{form-A1} holds.
\end{proof}

\end{document}